\documentclass[MSNbibl,number,citesort,seceqn,dvips]{arxbj}
\usepackage{graphicx}
\aid{0}
\volume{20}
\issue{2}
\pubyear{2014}
\firstpage{545}
\lastpage{585}
\doi{10.3150/12-BEJ497} %kopijuoti is PTS

\makeatletter
\newcommand{\rrVert}{\Vert}
\newcommand{\rrvert}{\vert}
\newcommand{\llVert}{\Vert}
\newcommand{\llvert}{\vert}
\newtheorem{theorem}{Theorem}[section]
\newtheorem{lemma}[theorem]{Lemma}
\newtheorem{corollary}[theorem]{Corollary}
\newtheorem{proposition}[theorem]{Proposition}
\newproclaim{algorithm}[theorem]{Algorithm}
\newproclaim{condition}[theorem]{Condition}
\newremark{remark}[theorem]{Remark}

\newcommand{\F}{\mathcal{F}}
\newcommand{\G}{\mathcal{G}}
\newcommand{\E}{\mathbb{E}}
\newcommand{\R}{\mathbb{R}}
\newcommand{\ud}{\mathrm{d}}
\renewcommand{\P}{\mathbb{P}}
\newcommand{\Var}{\operatorname{Var}}
\newcommand{\PF}{\operatorname{PF}}
\newcommand{\Hess}{\operatorname{Hess}}
\newcommand{\Poisson}{\operatorname{Poisson}}
\newcommand{\AR}{\operatorname{AR}}
\renewcommand{\emptyset}{\varnothing}
\renewcommand{\mid}{\vert}
\newcommand{\eqref}[1]{(\ref{#1})}
\newcommand{\xrightarrow}{\hbox to 30pt{\rightarrowfill}}

\makeatother

\begin{document}
\begin{frontmatter}

\title{Markovian stochastic approximation with expanding projections}
\runtitle{Markovian stochastic approximation with expanding projections}

\begin{aug}
%%%% inicialai - be tarpu
\author[1]{\fnms{Christophe} \snm{Andrieu}\thanksref{1}\ead[label=e1]{C.Andrieu@bristol.ac.uk}} \and
\author[2]{\fnms{Matti} \snm{Vihola}\corref{}\thanksref{2}\ead[label=e2]{matti.vihola@iki.fi}}
\runauthor{C. Andrieu and M. Vihola} %% auto
\address[1]{School of Mathematics, University of Bristol, BS8 1TW, United
Kingdom.\\ \printead{e1}}
\address[2]{Department of Mathematics and Statistics, University of
Jyv\"{a}skyl\"{a}, P.O. Box 35, FI-40014, Finland. \printead{e2}}
\end{aug}

% HISTORY:
\received{\smonth{11} \syear{2011}}
\revised{\smonth{9} \syear{2012}}

% ABSTRACT
%
\begin{abstract}
Stochastic approximation is a framework unifying many random iterative
algorithms occurring in a diverse range of applications. The stability of
the process is often difficult to verify in practical applications and the
process may even be unstable without additional stabilisation
techniques. We
study a stochastic approximation procedure with expanding projections
similar to Andrad\'{o}ttir [\textit{Oper. Res.} \textbf{43} (1995)
1037--1048]. We focus
on Markovian noise and show the stability and convergence under general
conditions. Our framework also incorporates the possibility to use a random
step size sequence, which allows us to consider settings with a non-smooth
family of Markov kernels. We apply the theory to stochastic
approximation expectation
maximisation with particle independent Metropolis--Hastings sampling.
\end{abstract}

% KEYWORDS
% visi is mazosios raides ir pagal abecele
%
\begin{keyword}
\kwd{expectation maximisation}
\kwd{independent Metropolis--Hastings}
\kwd{particle Markov chain Monte Carlo}
\kwd{stability}
\kwd{stochastic approximation}
\end{keyword}

\end{frontmatter}

%s1 #&#
\section{Introduction}\label{sec:intro} %{{{

Stochastic approximation (SA) is concerned with finding the zeros of a
function defined on the space $\Theta\subset\mathbb{R}^{d}$ as
%
%e1 #&#
\begin{equation}\label{eq:mean-field}
h(\theta):=\int_{\mathsf{X}}H(\theta,x)\pi_{\theta}(\ud x) ,
\end{equation}
where $\{\pi_{\theta}\}_{\theta\in\Theta}$ is a family of probability
distributions on a generic measurable space
$ (\mathsf{X},\mathcal{B}(\mathsf{X}) )$ and
$H\dvtx \Theta\times\mathsf{X}\to\Theta$
is a measurable function. In numerous situations $h$ behaves like
a gradient, suggesting that a recursion of the type $\theta_{i+1}=\theta_{i}+\gamma_{i+1}h(\theta_{i})$
where $(\gamma_{i})_{i\ge1}$ is a sequence of nonnegative step sizes
decaying to zero, can
be used to find the aforementioned roots.

Often in applications, the integral \eqref{eq:mean-field}
needs to be approximated numerically. We focus here on methods
relying on Monte Carlo simulation where sampling exactly from $\pi_{\theta}$
for any $\theta\in\Theta$ is not possible directly and instead Markov
chain Monte Carlo methods are used. Let
$\{P_{\theta}\}_{\theta\in\Theta}$ be a family
of Markov transition probabilities with stationary
distributions
$\{\pi_\theta\}_{\theta\in\Theta}$, respectively.
Then, the standard SA recursion with Markovian dynamic is as follows
\begin{eqnarray*}
X_{i+1}\mid\theta_0,X_0,\ldots,
\theta_i,X_i &\sim& P_{\theta_{i}}(X_{i},
\cdot ),
\\
\theta_{i+1}&=&\theta_{i}+\gamma_{i+1}H(
\theta_{i},X_{i+1}) .
\end{eqnarray*}

Stability of this process is far from obvious and a significant effort
has been dedicated to its study (e.g., \cite{benaim}, Section~7.3).
Problems occur in particular when
ergodicity, a term to be made more precise later, of $P_{\theta}$
vanishes as $\theta$ approaches a set of critical values denoted
$\partial\Theta$ hereafter. Younes \cite{younes}, Section~6.3,
gives an example of a situation where the Robbins--Monro algorithm fails for
this reason.

Cures include projection on a fixed
set $\mathcal{R}_{0}\subset\Theta$, that is, given
a projection mapping $\Pi_{\mathcal{R}_{0}}\dvtx \Theta\setminus
\mathcal{R}_0\rightarrow\mathcal{R}_{0}$,
one can define \cite{kushner-clark,kushner-yin-sa}
\begin{eqnarray*}
\theta_{i+1}^{\ast} & =&\theta_{i}+
\gamma_{i+1}H(\theta_{i},X_{i+1}),
\\
\theta_{i+1} & =&\theta_{i+1}^{\ast}\mathbb{I}\bigl
\{\theta_{i+1}^{\ast
}\in\mathcal{R}_{0}\bigr\}+
\Pi_{\mathcal{R}_{0}}\bigl(\theta_{i+1}^{\ast
}\bigr)\mathbb{I}
\bigl\{\theta_{i+1}^{\ast}\notin\mathcal{R}_{0}\bigr
\} .
\end{eqnarray*}
Projection on a fixed set $\mathcal{R}_{0}$ might not be satisfactory
when for example the location of the zeros of $h(\theta)$
is not known a priori. It is also possible that the projection induces
spurious attractors on the boundary of $\mathcal{R}_0$.

Adaptive projections overcome these difficulties
by considering an increasing sequence of
projection sets
$\{\mathcal{R}_{i}\}_{i\ge0}$ which forms
a covering of $\Theta$.
The process is defined through
\cite{chen-varying,chen-guo-gao,chen-sa,sa-verifiable,tadic-random-truncations}
\begin{eqnarray*}
\theta_{i+1}^{\ast} & =&\theta_{i}+
\gamma_{i+1}H(\theta_{i},X_{i+1}),
\\
\theta_{i+1} & =&\theta_{i+1}^{\ast}\mathbb{I}\bigl
\{\theta_{i+1}^{\ast
}\in\mathcal{R}_{r_{i}}\bigr\}+
\Pi_{\mathcal{R}_{0}}\bigl(\theta_{i+1}^{\ast
}\bigr)\mathbb{I}
\bigl\{\theta_{i+1}^{\ast}\notin\mathcal{R}_{r_{i}}\bigr
\},
\\
r_{i+1} & =&r_{i} +\mathbb{I}\bigl\{\theta_{i+1}^{\ast}
\notin\mathcal{R}_{r_{i}}\bigr\} ,
\end{eqnarray*}
where $r_i$ is the indicator of the current reprojection set and
$r_0\equiv0$.
Adaptive projections can be shown to lead
to stable recursions under rather general conditions.
In the case of a Markovian noise, one usually modifies
also $X_{i+1}$ so that \cite{sa-verifiable}
\begin{eqnarray*}
X_{i+1}\mid\theta_0,X_0,\ldots,
\theta_i,X_i &\sim& P_{\theta_i}
\bigl(X_i^*, \cdot \bigr) \qquad \mbox{with}
\\
X_i^* &\hspace*{2pt}:=&\mathbb{I}\bigl\{\theta_i^*\in
\mathcal{R}_{i-1}\bigr\} X_i +\mathbb{I}\bigl\{
\theta_i^*\notin\mathcal{R}_{i-1}\bigr\} \hat{
\Pi}_{\mathsf{K}_0}(X_i) ,
\end{eqnarray*}
where $\hat{\Pi}_{\mathsf{K}_0}\dvtx \mathsf{X}\to\mathsf{K}_0$ maps
$X_i$ to a
suitable (usually compact) set $\mathsf{K}_0\subset\mathsf{X}$. This
corresponds
effectively to `restarting' the process, with a smaller step size
sequence and a bigger
feasible set $\mathcal{R}_{r_i+1}$. One can show that the projections occur
finitely often under fairly general conditions, whence the process is
eventually stable \cite{sa-verifiable}. In practice, this algorithm
may be
wasteful if $\{\mathcal{R}_i\}_{i\ge0}$ or $\mathsf{K}_0$ are
ill-defined, and
the projections occur frequently.

We focus here on the study of a different stabilising approach where
projection occurs on an expanding (with time) sequence of projection
sets $\{\mathcal{R}_{i}\}$. Our approach is similar to
Andrad\'{o}ttir's \cite{andradottir}; see also
\cite{sharia-truncated,sharia-moving-bounds}, but we consider a more
general framework
with two major differences. First, we focus on a Markovian noise setting,
and second, we allow the step size sequence,
now denoted $(\Gamma_i)_{i\ge1}$, to
be random.\footnote{The recent work of Sharia
\cite{sharia-moving-bounds} includes random step sizes as well, but
our assumptions on $\Gamma_i$ are completely different.}
Our analysis is inspired by earlier related
work in adaptive Markov chain Monte Carlo \cite{saksman-vihola}.
The generic algorithm can be given as follows.
%%%%%%%%%%%%%%%%%%%%
%
%al1.1 #&#
\begin{algorithm}\label{alg:saep} %{{{
Let $\{\mathcal{R}_i\}_{i\ge0}$ be subsets of $\Theta$
and let the weights $(\Gamma_i)_{i\ge1}$ be nonnegative random variables.
The stochastic approximation process
$(\theta_{i},X_{i})_{i\ge0}$ with expanding projection
sets $\{\mathcal{R}_i\}_{i\ge0}$ is defined
for any starting point $(\theta_0,X_0)\equiv(\theta,x)\in
\mathcal{R}_0\times\mathsf{X}$ and
recursively for $i\geq0$ as follows
\begin{eqnarray*}
X_{i+1}\mid\mathcal{F}_{i} & \sim &P_{\theta_{i}}(X_{i},
\cdot ),
\\
\theta_{i+1}^{\ast} & =&\theta_{i}+
\Gamma_{i+1}H(\theta_{i},X_{i+1}),
\\
\theta_{i+1} & =& \theta_{i+1}^{\ast}\mathbb{I}\bigl
\{\theta_{i+1}^{\ast}\in\mathcal {R}_{i+1}\bigr\} +
\theta_{i+1}^{\mathrm{proj}} \mathbb{I}\bigl\{\theta_{i+1}^{\ast}
\notin\mathcal{R}_{i+1}\bigr\} ,
\end{eqnarray*}
where
$\mathcal{F}_{i}$ stands for the $\sigma$-algebra generated by
$\theta_0,X_{0},\theta_1,X_1,\Gamma_1,\ldots,\theta_i,X_i,\Gamma_i$,
and where $\theta_{i+1}^{\mathrm{proj}}$ is
a $\sigma(\F_i,X_{i+1},\theta_{i+1}^*)$-measurable random variable
taking values in $\mathcal{R}_{i+1}$.
\end{algorithm}
%
%}}}
Most common practical projection mechanisms include
$\theta_{i+1}^{\mathrm{proj}}:=\theta_i$ `rejecting' an update
outside the current feasible set, and $\theta_{i+1}^{\mathrm{proj}}
:=\Pi_{\mathcal{R}_{i+1}}(\theta_{i+1}^*)$, where
$\Pi_{\mathcal{R}_{i+1}}\dvtx \Theta\setminus\mathcal{R}_{i+1}\to
\mathcal{R}_{i+1}$
is a measurable mapping.

In words, the expanding projections approach only ensures that
$\theta_i$ is in a feasible set $\mathcal{R}_i$ but does not involve
potentially harmful
`restarts' as is the case with the adaptive reprojection strategy. Note
particularly
that unlike with the adaptive reprojections strategy,
we need not project $X_{i+1}$ at all.
We believe that these advantages can provide significantly
better results in certain settings, but this is at the expense of
requiring more when proving the stability and the convergence of the
process. In short, we must be able to control certain quantitative
criteria within each feasible set $\mathcal{R}_i$.
The random step size sequence allows one to consider
situations where the family of Markov kernels
$\{P_{\theta}\}_{\theta\in\Theta}$ is not necessarily smooth in a
manner that is usually considered in the stochastic approximation
literature (e.g., \cite{benveniste-metivier-priouret}).

Other stabilisation techniques in the literature related to our
approach include the state-dependent
averaging framework of Younes \cite{younes} and
a state-dependent step size sequence of Kamal \cite{kamal}.
Particularly the former shares similarities with
the present work, as it also relies on quantifying the ergodicity rates of
Markov kernels explicitly. Our stabilisation approach differs, however,
crucially from these methods, adding only the projections to the basic
Robbins--Monro algorithm. We remark also that our present approach
may be used in some situations to prove the stability and convergence of
an \emph{unmodified} Robbins--Monro stochastic approximation.
This is possible, loosely speaking,
if one can show that projections do not occur at all with a positive
probability; see \cite{saksman-vihola} for an example of such a
situation. We point out also the work
\cite{andrieu-tadic-vihola}
suggesting a generic method to establish the stability of
unmodified Markovian Robbins--Monro stochastic approximation
at the expense of more stringent assumptions.

Our main results show that the SA process $(\theta_{i})_{i\ge0}$
produced by our expanding projections algorithm `stays away from
$\partial\Theta$' almost surely for any starting point
$(\theta,x)\in\mathcal{R}_0\times\mathsf{X}$ under conditions on
$H(\cdot,\cdot)$, $\{P_{\theta}\}_{\theta\in\Theta}$,
$(\mathcal{R}_{i})_{i\ge0}$ and $(\Gamma_i)_{i\ge1}$. Figure~\ref
{fig:navigation}
summarises the inter-dependency between our various main conditions and
results and in order to help the reader we provide a nomenclature of
some of the constants involved in Appendix~\ref{sec:nom}.

Section~\ref{sec:stability} contains two fundamental results, Theorems
\ref{th:unbounded-w-stability} and \ref{th:bounded-w-stability},
which both establish stability of Algorithm \ref{alg:saep} under
abstract noise conditions and the existence of a Lyapunov function
satisfying two distinct sets of assumptions which, roughly speaking,
allow us to tackle instability at infinity or at a finite point.
Section~\ref{sec:noise} focuses on establishing the required
noise conditions with verifiable assumptions on the Markov kernels.
First, Theorem~\ref{thm:general-noise-theorem} establishes the
aforementioned noise conditions under Condition \ref
{cond:general-noise-theorem-cond}, which essentially involves a
trade-off between the sequences $(\Gamma_i)_{i \geq0}$ and $(\xi_i)_{i \geq0}$ and properties of the solution of the Poisson equation
related to $\{P_\theta\}_{\theta\in\Theta}$ and $H(\cdot,\cdot)$.
Second, essentially assuming geometric ergodicity, Propositions \ref
{prop:final-prop-cont} and \ref{prop:nonsmooth-noise} establish the
required conditions in the scenarios where $\{P_\theta\}_{\theta\in
\Theta}$ depends smoothly on $\theta$ and where it does not
respectively---the latter case requires the introduction of random
step-sizes $(\Gamma_i)_{i \geq0}$ (see also the comments in the
introduction of Section~\ref{sec:random-step-size}).
%
%f1 #&#
\begin{figure}

\includegraphics{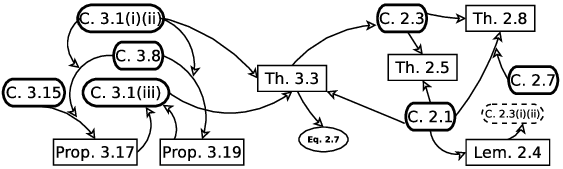}

\caption{Road map of the main results and assumptions.}
\label{fig:navigation}
\end{figure}

We complement our stability results in Section~\ref{sec:convergence}
with a discussion on how one can use existing results in the literature
to obtain convergence of $(\theta_i)_{i\ge0}$ to a zero of $h$.
Finally, we apply our theory to a new stochastic approximation
expectation maximisation algorithm involving particle independent
Metropolis--Hastings sampling in Section~\ref{sec:application}.

%}}}

%%%%%%%%%%%%%%%%%%%%%%%%%%%%%%%%%%%%%%%%%%%%%%%%%%%%%%%%%%%%%%%%%%%%%%%%%%%%%%%%%
%s2 #&#
\section{General stability results}\label{sec:stability} %{{{

%{{{

We denote throughout the article
the probability distribution associated to the process
$(\theta_i,X_i)_{i\ge0}$ defined in Algorithm
\ref{alg:saep} and starting at
$(\theta_0,X_0)\equiv(\theta,x)\in\Theta\times\mathsf{X}$
as $\P_{\theta,x}( \cdot )$
and the associated expectation as $\E_{\theta,x}[ \cdot ]$.
For any subset $A\subset E$
of some space $E$, we denote $A^{c}$ its complement in $E$.
We also denote $ \langle \cdot , \cdot  \rangle$
the standard inner product and $| \cdot |$ the associated norm on
$\Theta\subset\R^d$.
We also use the notation $a\vee b:=\max\{a,b\}$ and $a\wedge
b :=\min\{a,b\}$.

The approach we develop relies on the existence of a Lyapunov function
$w\dvtx \Theta\rightarrow[0,\infty)$ for the recursion on $\theta$ and the
subsequent proof that $\{w(\theta_{i})\}$ is $\P_{\theta,x}$-a.s.
under some adequate level.
For any $M>0$, we define the level sets
$\mathcal{W}_{M}:=\{\theta\in\Theta\dvt w(\theta)\leq M\}$.
Our general stability results are inspired by a proof
due to Benveniste, Metivier and Priouret \cite{benveniste-metivier-priouret}, Theorem~17, page 239, but differ in many respects
as we shall see.

We consider two different settings concerning the way $w$ behaves on
the boundary
$\partial\Theta$ of $\Theta$. Section~\ref{sec:unbounded-lyapunov}
assumes that $\lim_{\theta\rightarrow\partial\Theta}w(\theta
)=\infty$,
which is well suited for example to the case $\Theta=\R$ and
$\partial\Theta=\{-\infty,\infty\}$. Section~\ref{sec:bounded-lyapunov} considers the case where $w$ may not be
unbounded, which requires stronger assumptions on the behaviour of $w$.
This setting subsumes for example the case where $\Theta\subset\R$ and
$\partial\Theta$ contains some points on the real line. Both of the
scenarios share the following set of assumptions.
%%%%%%%%%%%%%%%%%%%%
%
%co2.1 #&#
\begin{condition}\label{cond:w} %{{{
There exists a twice continuously differentiable
function $w\dvtx \Theta\rightarrow[0,\infty)$ such that
\renewcommand\theenumi{(\roman{enumi})}
\renewcommand\labelenumi{\theenumi}
\begin{enumerate}[(iii)]
\item\label{item:w-hessian} the Hessian matrix
$\Hess_w\dvtx \Theta\to\R^{d\times d}$ of $w$ is bounded so that
\begin{eqnarray*}
C_{w}:=\sup_{\theta\in\Theta} \sup_{|\theta_0|=1} \bigl|
\Hess_w(\theta)\theta_0\bigr|<\infty,
\end{eqnarray*}
\item\label{item:projection-sets}
the projection sets are increasing subsets of $\Theta$, that is,
$\mathcal{R}_{i}\subset\mathcal{R}_{i+1}$
for all $i\geq0$, and $\hat{\Theta}:=\bigcup_{i=0}^\infty
\mathcal{R}_i \subset\Theta$,
\item\label{item:w-drift}
there exists a constant $M_{0}>0$ such that for any
$\theta\in\mathcal{W}_{M_{0}}^{c}\cap\hat{\Theta}$
\[
\bigl\langle\nabla w(\theta),h(\theta) \bigr\rangle\leq0 ,
\]
\item\label{item:proj} the family of random variables
$\{\theta_{i}^{\mathrm{proj}}\}_{i\ge1}$
satisfies for all $i\geq1$ whenever
$\theta_{i}^{\ast}\notin\mathcal{R}_{i}$
\begin{eqnarray*}
\theta_{i}^{\mathrm{proj}} \in\mathcal{R}_i \quad \mbox{and}\quad  w
\bigl(\theta_{i}^{\mathrm{proj}}\bigr)\le w\bigl(\theta^{\ast}_i
\bigr)\qquad  \P_{\theta,x}\mbox{-a.s.} ,
\end{eqnarray*}
\item\label{item:xi}
there exists constants $\alpha_w,c\in[0,\infty)$ and
a non-decreasing sequence of constants $\xi_i\in[1,\infty)$ satisfying
$\sup_{\theta\in\mathcal{R}_{i}}|\nabla w(\theta)|\le c\xi_i^{\alpha_w}$
for all $i\ge0$.
\end{enumerate}
%
%}}}
\end{condition}
%%%%%%%%%%%%%%%%%%%%
%
%re2.2 #&#
\begin{remark} %{{{
Condition \ref{cond:w}
\begin{enumerate}[(iii)]
\item[\ref{item:w-hessian}]
Can often be established by introducing a Lyapunov function
defined through $w:=\psi\circ\tilde{w}$, where
$\psi\dvtx [0,\infty)\to[0,\infty)$ is a suitable concave function
modifying the values of another Lyapunov function
$\tilde{w}$ which satisfies the drift condition
\ref{item:w-drift} but does not have finite second
derivatives; see \cite{benveniste-metivier-priouret}, Remark on page 239.
\item[\ref{item:projection-sets}]
Is often satisfied with $\hat{\Theta}=\Theta$, but
accomodates also projections sets which do not cover $\Theta$, but
only certain admissible values
$\hat{\Theta}\subsetneq\Theta$.
As an extreme case, this allows to use the present framework
to check that a
fixed projection does not induce spurious attractors on the boundary
of $\hat{\Theta}$. Notice also that the function $H(\theta,x)$ and the
corresponding mean field $h(\theta)$ need only be defined for values
$\theta\in\hat{\Theta}$.
\item[\ref{item:w-drift}]
Will be replaced with a stricter drift in Theorem~\ref{th:bounded-w-stability}, where $w$ is not required to diverge
on the boundary $\partial\hat{\Theta}$.
\item[\ref{item:proj}]
Is satisfied trivially by
the choices $\theta_{i}^{\mathrm{proj}}
:=\theta_{i-1}$ and
$\theta_{i}^{\mathrm{proj}}:=
\Pi_{\mathcal{R}_i}(\theta_i^*)$,
if the projection sets
are defined as the level sets of the Lyapunov function, that is
$\mathcal{R}_i:=\mathcal{W}_{M_i}$ for some $M_i>0$.
In the Markovian case, the projections are assumed to satisfy
an additional continuity condition; see Theorem~\ref
{thm:general-noise-theorem}.
\item[\ref{item:xi}] Involves in practice
a sequence that grows at most at a rate $\xi_i :=i\vee1$, with
some power $\alpha_w\in[0,1)$. The sequence $\xi_i$ plays a central
role also in controlling the ergodicity rate of the Markov chain in
$\mathcal{R}_i$; see Remark~\ref{rem:discussion-of-constants}.
\end{enumerate}
\end{remark}
%
%}}}

Hereafter, we denote the `centred' version of $H$ as
$\bar{H}(\theta,x):=H(\theta,x)-h(\theta)$. For the stability
results, we shall introduce the following general condition on the noise
sequence. In general terms, it is
related to the rate at which $\{\theta_{i}\}$
may approach $\partial\hat{\Theta}$ in relation to the growth of
$|H(\theta,x)|$
and the loss of ergodicity of $\{P_{\theta}\}$. Establishing practical
and realistic conditions under which this assumption holds will be
the topic of Section~\ref{sec:noise}.
%%%%%%%%%%%%%%%%%%%%
%
%co2.3 #&#
\begin{condition}\label{cond:abstract-growth} %{{{
For any
$(\theta,x)\in\mathcal{R}_0\times\mathsf{X}$ it holds that
\renewcommand\theenumi{(\roman{enumi})}
\renewcommand\labelenumi{\theenumi}
\begin{enumerate}[(iii)]
\item\label{eq:incr-vanish}\hspace*{4pt} $\P_{\theta,x} \Bigl( \displaystyle \lim_{i\to\infty} \Gamma_{i+1} \bigl|\nabla w(
\theta_i)\bigr| \cdot\bigl| H(\theta_i,X_{i+1})\bigr| = 0
\Bigr)=1$,\vspace*{3.5pt}

\item\label{eq:square-summable}\hspace*{4pt} $\E_{\theta,x} \Biggl[{\displaystyle \sum_{i=0}^\infty}
\Gamma_{i+1}^{2}\bigl|H(\theta_{i},X_{i+1})\bigr|^{2}
\Biggr] <\infty$,\vspace*{3.5pt}

\item\label{eq:noise-summable}\hspace*{4pt} $\E_{\theta,x} \Biggl[\displaystyle \sup_{k\geq0} \Biggl\llvert \sum
_{i=0}^{k}\Gamma_{i+1} \bigl\langle\nabla
w(\theta_{i}),\bar{H}(\theta_{i},X_{i+1}) \bigr
\rangle\Biggr\rrvert \Biggr] <\infty$.
\end{enumerate}
%
%}}}
\end{condition}
In what follows, we shall focus on
a single condition implying Condition \ref{cond:abstract-growth}\ref{eq:incr-vanish} and \ref{eq:square-summable}.
It is slightly more stringent, but more convenient to
check in practice.
%%%%%%%%%%%%%%%%%%%%
%
%le2.4 #&#
\begin{lemma} %{{{
Suppose Condition \ref{cond:w} holds and
%
%e5 #&#
\begin{equation}\label{eq:square-summable-alt}
\E_{\theta,x} \Biggl[{\sum_{i=0}^\infty}
\Gamma_{i+1}^{2}\xi_i^{2\alpha_w} \bigl|H(
\theta_{i},X_{i+1})\bigr|^{2} \Biggr] <\infty.
\end{equation}
Then, Condition \ref{cond:abstract-growth}\textup{\ref{eq:incr-vanish}} and
\textup{\ref{eq:square-summable}} hold.
\end{lemma}
%
%}}}
%
\begin{pf} %{{{
Note first that Condition \ref{cond:abstract-growth}\ref{eq:square-summable} holds trivially, because
$\xi_i^{2\alpha_w}\ge1$.
For Condition \ref{cond:abstract-growth}\ref{eq:incr-vanish}, consider
\[
\E_{\theta,x} \Biggl[{ \sum_{i=0}^\infty}
\bigl( \Gamma_{i+1}\bigl|\nabla w(\theta_i)\bigr|\cdot\bigl|H(
\theta_i,X_{i+1})\bigr| \bigr)^2 \Biggr] \le
c^2 \E_{\theta,x} \Biggl[{ \sum_{i=0}^\infty}
\Gamma_{i+1}^2\xi_{i}^{2\alpha_w} \bigl|H(
\theta_i,X_{i+1})\bigr|^2 \Biggr] .
\]\upqed
\end{pf}
%
%}}}

%}}}

%%%%%%%%%%%%%%%%%%%%%%%%%%%%%%%%%%%%%%%%
%s2.1 #&#
\subsection{Unbounded Lyapunov function}\label{sec:unbounded-lyapunov} %{{{

When $\lim_{\theta\to\partial\hat{\Theta}} w(\theta)=\infty$,
it is enough to show that the sequence $w(\theta_i)$ is bounded in
order to ensure the stability of $\theta_i$.
%%%%%%%%%%%%%%%%%%%%
%
%th2.5 #&#
\begin{theorem}\label{th:unbounded-w-stability} %{{{
Assume Conditions \ref{cond:w} and
\ref{cond:abstract-growth} hold.
Then, for any $(\theta,x)\in\mathcal{R}_0\times\mathsf{X}$
\[
\P_{\theta,x} \Bigl(\limsup_{i\to\infty} w(\theta_i) <
\infty \Bigr)=1 .
\]
\end{theorem}
%
%}}}
%
\begin{pf} %{{{
To show the $\P_{\theta,x}$-a.s. boundedness of
$\{w(\theta_{i})\}$ we fix
$(\theta,x)\in\mathcal{R}_0\times\mathsf{X}$ and
introduce the following quantities. Let
$M_0<M_1<\cdots<M_n\to\infty$
be an increasing sequence tending to infinity
and consider the level sets
$\mathcal{W}_{M_{i}}\subset\Theta$.
We assume that
$M_0$ is chosen large enough so that
$\theta_{0}=\theta\in\mathcal{W}_{M_{0}}$. For
any $n\geq0$, we define the first exit time of $\theta_i$
from the level set $\mathcal{W}_{M_{n}}$ as
\[
\sigma_{n}:=\inf\{i\ge0 \dvt  \theta_{i}\notin\mathcal
{W}_{M_{n}}\} ,
\]
with the usual convention that $\inf\{\emptyset\}=\infty$. For
any $n\geq0$, we define the time following the last exit of $\theta_i$
from $\mathcal{W}_{M_{0}}$
before $\sigma_{n}$ as
\[
\tau_{n}:= 1+\sup\{i\le\sigma_n \dvt  \theta_{i}
\in\mathcal{W}_{M_{0}}\} ,
\]
which is finite at least whenever $\sigma_n$ is finite by our
assumption that
$\theta_{0}\in\mathcal{W}_{M_{0}}$.
With these definitions, the claim holds once we show that
$\lim_{n\to\infty} \P_{\theta,x}(\sigma_n<\infty)=0$.

To begin with, define for $n\ge1$
the following sets characterising the jumps out of
$\mathcal{W}_{M_0}$
\begin{eqnarray*}
D_{n}:= \biggl\{ \mathbb{I}\{\tau_n<\infty\} \bigl[w(
\theta_{\tau_{n}})-w(\theta_{\tau
_{n}-1}) \bigr] \leq\frac{M_{n}-M_{0}}{2}
\biggr\} .
\end{eqnarray*}
We first show that $\lim_{n\to\infty}\P_{\theta,x}(D_n)=1$. Clearly
%
%e6 #&#
\begin{equation}\label{eq:trivial-superset}
\tilde{D}_n := \biggl\{ \sup_{i\geq0} \bigl[w(
\theta_{i+1})-w(\theta_{i}) \bigr]\le\frac
{M_{n}-M_{0}}{2} \biggr
\} \subset D_n
\end{equation}
and since $M_n\to\infty$, one has
$ \{ \sup_{i\geq0} [w(\theta_{i+1})-w(\theta_{i}) ]
<\infty \} = \bigcup_{n=1}^\infty\tilde{D}_n$.
Lemma~\ref{lem:D_n_has_prob_1} shows that
$1 = \P_{\theta,x}(\bigcup_{n=1}^\infty\tilde{D}_{n})
= \lim_{n\rightarrow\infty}\P_{\theta,x}(\tilde{D}_n)
\le\lim_{n\rightarrow\infty}\P_{\theta,x}(D_n)$ because
$\tilde{D}_n$ is an increasing sequence and by
\eqref{eq:trivial-superset}, respectively.

Now, it remains to focus on proving that
\[
\lim_{n\rightarrow\infty} \P_{\theta,x} \bigl(D_{n}\cap\{
\sigma_{n}<\infty\} \bigr)=0 .
\]
In order to achieve this observe first that
$w(\theta_{\sigma_{n}})-w(\theta_{\tau_{n}-1})\geq M_{n}-M_{0}$
on $\{\sigma_{n}<\infty\},$ implying that
on $D_{n}\cap\{\sigma_{n}<\infty\}$,
\[
w(\theta_{\sigma_{n}})-w(\theta_{\tau_{n}}) =w(\theta_{\sigma_{n}})-w(
\theta_{\tau_{n}-1}) - \bigl[w(\theta_{\tau_{n}})-w(\theta_{\tau_{n}-1})
\bigr] \geq\frac{M_{n}-M_{0}}{2} .
\]
This allows us to deduce the following bound
\begin{eqnarray*}
\P_{\theta,x}\bigl(D_{n}\cap\{\sigma_{n} <\infty\}
\bigr)& =&\E_{\theta,x} \bigl[\mathbb{I}\bigl\{D_{n}\cap\{
\sigma_{n}<\infty\}\bigr\} \bigr]
\\
& \leq&\E_{\theta,x} \biggl[\mathbb{I}\bigl\{D_{n}\cap\{
\sigma_{n}<\infty \}\bigr\}\frac{w(\theta_{\sigma_{n}})-w(\theta_{\tau_{n}})}{
({1}/{2})(M_{n}-M_{0})} \biggr]
\\
& \le&\frac{2}{M_{n}-M_{0}} \E_{\theta,x} \bigl[ \mathbb{I}\{
\sigma_{n}<\infty\} \bigl[w(\theta_{\sigma_{n}})-w(
\theta_{\tau_{n}}) \bigr] \bigr] .
\end{eqnarray*}
Since $M_n\to\infty$,
the proof will be finished once we show that
%
%e7 #&#
\begin{equation}\label{eq:w-exp-bound}
\sup_{n\geq0} \E_{\theta,x} \bigl[ \mathbb{I}\{\sigma_{n}<
\infty\} \bigl[w(\theta_{\sigma_{n}})-w(\theta_{\tau_{n}})\bigr] \bigr]<
\infty.
\end{equation}

Thanks to Condition \ref{cond:w}\ref{item:proj},
we have for any $i\ge0$
that $w(\theta_{i+1})\le w(\theta_{i+1}^*)$ and consequently
\begin{eqnarray*}
w(\theta_{i+1})-w(\theta_i) &\le& \Gamma_{i+1}
\bigl\langle\nabla w(\theta_{i}),h(\theta_{i}) \bigr\rangle
\\
&&{}+\Gamma_{i+1} \bigl\langle\nabla w(\theta_{i}),\bar{H}(
\theta_{i},X_{i+1}) \bigr\rangle +\Gamma_{i+1}^2
\frac{C_{w}}{2}\bigl|H(\theta_{i},X_{i+1})\bigr|^{2} .
\end{eqnarray*}
So in particular, since
$\langle\nabla w(\theta_i), h(\theta_i) \rangle\le0$
whenever $\theta_i\in\mathcal{W}_{M_0}^c$,\vspace*{-1pt}
\begin{eqnarray*}
&&\mathbb{I}\{\sigma_{n}<\infty\} \bigl[w(\theta_{\sigma_{n}})-w(
\theta_{\tau_{n}})\bigr]\\
&&\quad  = \mathbb{I}\{\sigma_{n} <\infty\}{\sum
_{i=\tau_{n}}^{\sigma_{n}-1}} \bigl[w(\theta_{i+1})-w(
\theta_{i})\bigr]
\\
&&\quad \leq\mathbb{I}\{\sigma_{n}<\infty\} \Biggl( { \sum
_{i=\tau_{n}}^{\sigma_{n}-1}}\Gamma_{i+1} \bigl\langle\nabla
w(\theta_{i}),\bar{H}(\theta_{i},X_{i+1}) \bigr
\rangle +\Gamma_{i+1}^2\frac{C_{w}}{2}\bigl|H(
\theta_{i},X_{i+1})\bigr|^{2} \Biggr) .
\end{eqnarray*}
Recall the following estimate for partial sums\vspace*{-1pt}
%
%e8 #&#
\begin{equation}\label{eq:partial-sums}
\Biggl\llvert \sum_{i=j}^k
a_i\Biggr\rrvert = \Biggl\llvert \sum_{i=0}^k
a_i - \sum_{i=0}^{j-1}
a_i\Biggr\rrvert \le \Biggl\llvert \sum
_{i=0}^k a_i\Biggr\rrvert + \Biggl
\llvert \sum_{i=0}^{j-1} a_i
\Biggr\rrvert \le2 \sup_{k\ge0} \Biggl\llvert \sum
_{i=0}^k a_i \Biggr\rrvert ,
\end{equation}
implying in our case that\vspace*{-1pt}
\begin{eqnarray*}
&&\frac{1}{2} \mathbb{I}\{\sigma_{n}<\infty\} \bigl[w(
\theta_{\sigma_n})-w(\theta_{\tau_n})\bigr]
\\
&&\quad \le\mathbb{I}\{\sigma_{n}<\infty\} \Biggl( \sup_{k\ge0}
\Biggl\llvert \sum_{i=0}^k
\Gamma_{i+1} \bigl\langle\nabla w(\theta_{i}),\bar{H}(
\theta_{i},X_{i+1}) \bigr\rangle \Biggr\rrvert + \sum
_{i=0}^\infty\Gamma_{i+1}^2
\frac{C_{w}}{2}\bigl|H(\theta_{i},X_{i+1})\bigr|^{2}
\Biggr) .
\end{eqnarray*}
Now, Condition \ref{cond:abstract-growth}\ref{eq:square-summable} and
\ref{eq:noise-summable}
imply
\eqref{eq:w-exp-bound} allowing us to conclude.
\end{pf}
%
%}}}
%%%%%%%%%%%%%%%%%%%%
%
%le2.6 #&#
\begin{lemma}\label{lem:D_n_has_prob_1} %{{{
Under Condition \ref{cond:abstract-growth} we have,
$\P_{\theta,x}$-almost surely\vspace*{-1pt}
%
%e9 #&#
%e10 #&#
\begin{eqnarray}
\label{eq:jumps-vanish}\limsup_{i\to\infty} \bigl[w(\theta_{i+1})-w(\theta_i)
\bigr] &\le&0,
\\
\label{eq:jumps-finite}\sup_{i\geq0} \bigl[w(\theta_{i+1})-w(\theta_{i})
\bigr] &<&\infty .
\end{eqnarray}
\end{lemma}
%
%}}}
%
\begin{pf} %{{{
We first prove that $
\lim_{i\rightarrow\infty}|w(\theta_{i+1}^{\ast})-w(\theta_{i})|=0$,
$\P_{\theta,x}$\mbox{-a.s.}
By a Taylor expansion, we get\vspace*{-1pt}
\begin{eqnarray*}
\bigl|w\bigl(\theta_{i+1}^{\ast}\bigr)-w(\theta_{i})\bigr|
\le \bigl|\nabla w(\theta_i)\bigr| \cdot\bigl|\Gamma_{i+1} H(
\theta_i,X_{i+1})\bigr| + \Gamma_{i+1}^2
C_w \bigl|H(\theta_i, X_{i+1})\bigr|^2.
\end{eqnarray*}
The terms on the right converge to zero $\P_{\theta,x}$-a.s. by
Condition \ref{cond:abstract-growth}\ref{eq:incr-vanish}
and \ref{eq:square-summable}, respectively.
Now, \eqref{eq:jumps-vanish}
follows since by Condition \ref{cond:w}\ref{item:proj}
$w(\theta_{i+1})-w(\theta_i) \le w(\theta_{i+1}^*) - w(\theta_i)$.
We conclude by noting that \eqref{eq:jumps-finite} follows directly from
\eqref{eq:jumps-vanish}.
\end{pf}
%
%}}}

%}}}

%%%%%%%%%%%%%%%%%%%%%%%%%%%%%%%%%%%%%%%%
%s2.2 #&#
\subsection{Bounded Lyapunov function}\label{sec:bounded-lyapunov} %{{{

In the previous section, the Lyapunov function satisfied
$\lim_{\theta\to\partial\hat{\Theta}}w(\theta)=\infty$. If this
is not the
case, we need to replace Condition
\ref{cond:w}\ref{item:w-drift} with a more stringent condition
quantifying the drift outside $\mathcal{W}_{M_0}$, while not requiring
$\lim_{\theta\to\partial\hat{\Theta}} w(\theta) = \infty$.
%%%%%%%%%%%%%%%%%%%%
%
%co2.7 #&#
\begin{condition}\label{cond:stringent-drift} %{{{
The Lyapunov function and the step size sequence satisfy
%with Condition \ref{cond:w} holds, and additionally
%with \eqref{item:w-drift} replaced by a
%more stringent condition
%
\begin{eqnarray*}
\delta_{i}:=\inf_{\theta\in\mathcal{R}_{i}\setminus
\mathcal{W}_{M_{0}}} - \bigl\langle\nabla w(\theta),h(
\theta) \bigr\rangle >0 \quad \mbox{and}\quad  \sum_{i=1}^{\infty}
\Gamma_{i}\delta_{i}=\infty \qquad \P_{\theta,x}\mbox{-almost
surely}.
\end{eqnarray*}
\end{condition}
%
%}}}
%%%%%%%%%%%%%%%%%%%%
%
%th2.8 #&#
\begin{theorem}\label{th:bounded-w-stability} %{{{
Assume Conditions \ref{cond:w}, \ref{cond:abstract-growth}
and \ref{cond:stringent-drift} hold,
and in addition that the following condition on the noise holds
%
%e11 #&#
\begin{equation}\label{eq:noise-negligible}
\lim_{m\to\infty} \sup_{k>m} \Biggl\llvert \sum
_{i=m}^{k}\Gamma_{i+1} \bigl\langle\nabla
w(\theta_{i}),\bar{H}(\theta_{i},X_{i+1}) \bigr
\rangle\Biggr\rrvert = 0 .
\end{equation}
Then for any $M>M_0$, the tails of the
trajectories of $\{\theta_i\}$ are eventually contained within
$\mathcal{W}_M$
$\P_{\theta,x}$-a.s., that is,
\[
\P_{\theta,x}\biggl(\bigcup_{m\geq0}\bigcap
_{n\geq m} \{\theta_{n}\in
\mathcal{W}_{M}\}\biggr)=1 .
\]
\end{theorem}
%
%}}}
%
\begin{pf} %{{{
We first show that $\theta_{n}$ must visit
$\mathcal{W}_{M_{0}}$
infinitely often $\P_{\theta,x}$-a.s., in other words
%
%e12 #&#
\begin{equation}\label{eq:w-m0-io}
\P_{\theta,x} \biggl({ \bigcup_{m\geq1}} {
\bigcap_{n\geq
m}}\{\theta_{n}\notin
\mathcal{W}_{M_{0}}\} \biggr)=0 .
\end{equation}
For any $m\geq0$, we define
the hitting times
$\kappa_{m}:=\inf\{i>m\dvt\theta_{i}\in\mathcal{W}_{M_{0}}\}$
and notice that
\[
{\bigcup_{m\geq1}} {\bigcap
_{n\geq m}} \{\theta_{n}\notin\mathcal{W}_{M_{0}}
\} ={ \bigcup_{m\geq1}}\{\theta_{m}\notin
\mathcal{W}_{M_{0}}\} \cap\{\kappa_{m}=\infty\} .
\]
Recall that for any $i\geq0$
\begin{eqnarray*}
w(\theta_{i+1})-w(\theta_{i}) &\leq&\Gamma_{i+1}
\bigl\langle\nabla w(\theta_{i}), h(\theta_{i}) \bigr\rangle
\\
&&{} +\Gamma_{i+1} \bigl\langle\nabla w(\theta_{i}),\bar{H}(
\theta_{i},X_{i+1}) \bigr\rangle +\Gamma_{i+1}^2
\frac{C_{w}}{2}\bigl|H(\theta_{i},X_{i+1})\bigr|^{2} .
\end{eqnarray*}
So in particular, and thanks to Condition
\ref{cond:stringent-drift}, for
$n>m$
\begin{eqnarray*}
&&\mathbb{I}\{\theta_{m} \notin\mathcal{W}_{M_{0}}\} \bigl[w(
\theta_{n\wedge\kappa_{m}})-w(\theta_{m}) \bigr]
\\
&&\quad  =\mathbb{I}\{\theta_{m}\notin\mathcal{W}_{M_{0}}\} {\sum
_{i=m}^{(n\wedge\kappa_{m})-1}} \mathbb{I}\{
\theta_{i}\notin\mathcal{W}_{M_{0}}\} \bigl[w(
\theta_{i+1})-w(\theta_{i}) \bigr]
\\
&&\quad  \leq\mathbb{I}\{\theta_{m}\notin\mathcal{W}_{M_{0}}\} {
\sum_{i=m}^{(n\wedge\kappa_{m})-1}} \Gamma_{i+1}
\biggl[-\delta_{i}+ \bigl\langle\nabla w(\theta_{i}),
\bar{H}(\theta_{i},X_{i+1}) \bigr\rangle +
\Gamma_{i+1}\frac{C_{w}}{2}\bigl|H(\theta_{i},X_{i+1})\bigr|^{2}
\biggr] .
\end{eqnarray*}
From this, we obtain the following inequality holding
$\P_{\theta,x}$-a.s. on $\{\theta_{m}\notin\mathcal{W}_{M_{0}}\}$
for any $n>m$
%
%e13 #&#
\begin{eqnarray}\label{eq:impossible-ineq}
&&\E_{\theta,x} \Biggl[ \mathbb{I}\{\kappa_m=\infty\} {\sum
_{i=m}^{n-1}} \Gamma_{i+1}
\delta_{i} \Bigl| \mathcal{F}_{m} \Biggr] - w(
\theta_{m})\nonumber\\
&&\quad \le \E_{\theta,x} \Biggl[ \mathbb{I}\{\kappa_m=\infty\} {
\sum_{i=m}^{n-1}} \Gamma_{i+1}
\bigl\langle\nabla w(\theta_{i}),\bar{H}(\theta_{i},X_{i+1})
\bigr\rangle \\
&&\quad \hphantom{\le \E_{\theta,x} \Biggl[ \mathbb{I}\{\kappa_m=\infty\} {
\sum_{i=m}^{n-1}}}{}+\Gamma_{i+1}^2\frac{C_{w}}{2}\bigl|H(
\theta_{i},X_{i+1})\bigr|^{2} \Bigl| \mathcal{F}_{m}
\Biggr] .
\nonumber
\end{eqnarray}
Using this inequality, we shall see that for any $m>0$
%
%e14 #&#
\begin{equation}\label{eq:theta-notin-m0-forgood}
\P_{\theta,x} \bigl(\{\theta_{m}\notin\mathcal{W}_{M_{0}}
\} \cap\{\kappa_{m}=\infty\} \bigr)=0 .
\end{equation}
Suppose the contrary, $\P_{\theta,x}(\{\theta_{m}\notin\mathcal
{W}_{M_{0}}\}
\cap\{\kappa_{m}=\infty\})> 0$.
Then, because of Condition \ref{cond:stringent-drift},
we observe that the conditional expectation on the
left hand side of
\eqref{eq:impossible-ineq}
necessarily tends to infinity almost surely as $n\to\infty$.
Denote then the conditional expectation on the
right hand side of \eqref{eq:impossible-ineq}
by $E_{\theta,x}^{(m,n)}$.
As in the proof of Theorem~\ref{th:unbounded-w-stability}, we
have the following upper bound
\[
\E_{\theta,x}\bigl[E_{\theta,x}^{(m,n)}\bigr] \le
\E_{\theta,x} \Biggl[\sup_{k\ge0} \Biggl\llvert \sum
_{i=0}^k \Gamma_{i+1} \bigl\langle\nabla
w(\theta_{i}),\bar{H}(\theta_{i},X_{i+1}) \bigr
\rangle \Biggr\rrvert + \sum_{i=0}^\infty
\Gamma_{i+1}^2\frac{C_{w}}{2}\bigl|H(\theta_{i},X_{i+1})\bigr|^{2}
\Biggr] ,
\]
which is finite by Condition \ref{cond:abstract-growth} and
independent of $m$ and $n$.
By letting $n\to\infty$ we end up with a
contradiction, unless
\eqref{eq:theta-notin-m0-forgood} holds.
Consequently, the event
\[
{\bigcup_{m\geq1}} \{\theta_{m}\notin
\mathcal{W}_{M_{0}}\}\cap\{\kappa_{m}=\infty\}
\]
has null probability and we obtain \eqref{eq:w-m0-io}.

We now show that for any fixed $M>M_{0}$
\[
\P_{\theta,x} \biggl( \bigcup_{m\geq0}\bigcap
_{n\geq m} \{\theta_{n}\in
\mathcal{W}_{M}\} \biggr)=1 .
\]
We are going to apply Lemma~\ref{lemma:a-n} below with
$\delta= M-M_0>0$ to the events
\[
A_{m} =\{\theta_{m}\in\mathcal{W}_{M_{0}}\}
\cap{ \bigcup_{k>m}}\{\theta_{k} \notin
\mathcal{W}_{M}\} ,
\]
and denote
\[
B_{m}:=\{\theta_{m}\in\mathcal{W}_{M_{0}}\}
\setminus A_m = \{\theta_{m}\in\mathcal{W}_{M_{0}}
\} \cap{ \bigcap_{k>m}}\{\theta_{k}\in
\mathcal{W}_{M}\} .
\]
We may write
\begin{eqnarray*}
{ \bigcap_{n\geq1}} {\bigcup
_{m\geq n}}\{\theta_{m} \in\mathcal{W}_{M_{0}}
\} & =&{ \bigcap_{n\geq1}} {\bigcup
_{m\geq n}}A_{m}\cup B_{m}
\\
& =&{ \bigcap_{n\geq1}} \biggl[ \biggl({\bigcup
_{m\geq n}}A_{m} \biggr) \cup \biggl({\bigcup
_{m\geq n}}B_{m} \biggr) \biggr] .
\end{eqnarray*}
Now, since ${\bigcup_{m\geq n}}A_{m}$ and
${\bigcup_{m\geq n}}B_{m}$ are both decreasing events
with respect to $n\to\infty$, we have
\begin{eqnarray*}
1 &=& \lim_{n\rightarrow\infty}\P_{\theta,x} \biggl({\bigcup
_{m\geq n}} \{\theta_{m}\in\mathcal{W}_{M_{0}}
\} \biggr)
\\
&=&\lim_{n\rightarrow\infty} \biggl[ \P_{\theta,x} \biggl({ \bigcup
_{m\geq n}}A_{m} \biggr) +\P_{\theta,x} \biggl({
\bigcup_{m\geq n}}B_{m} \biggr) -
\P_{\theta,x} \biggl({ \bigcup_{m\geq n}}A_{m}
\cap{ \bigcup_{m\geq
n}}B_{m} \biggr) \biggr]
.
\end{eqnarray*}
By Lemma~\ref{lemma:a-n},
$\lim_{n\to\infty} \P_{\theta,x} (\bigcup_{m\ge n} A_m )
= 0$, so we end up with
$\lim_{n\rightarrow\infty}
\P_{\theta,x} (\bigcup_{m\geq n}B_{m} )=1$,
implying the claim.
\end{pf}
%
%}}}
%%%%%%%%%%%%%%%%%%%%
%
%le2.9 #&#
\begin{lemma}\label{lemma:a-n} %{{{
Assume the conditions of Theorem~\ref{th:bounded-w-stability}, let
$\delta>0$
and denote
\[
A_m := \{\theta_{m}\in\mathcal{W}_{M_{0}}\}
\cap{\bigcup_{k>m}}\{\theta_{k}\notin
\mathcal{W}_{M_0+\delta}\} .
\]
Then, $\lim_{n\to\infty} \P_{\theta,x} (\bigcup_{m\ge n}
A_m )=0$.
\end{lemma}
%
%}}}
%
\begin{pf} %{{{
Define the random times
$\sigma_m:=\inf\{i>m\dvt\theta_i\notin\mathcal{W}_{M_0+\delta}\}$
and $\tau_m:=\sup\{i\in[m,\sigma_m)\dvt
\theta_i\in\mathcal{W}_{M_0}\}+1$,
both finite on $A_m$.
Recall that on $\{\theta_i\in\mathcal{W}_{M_0}^c\}$ we have
\begin{eqnarray*}
w(\theta_{i+1})-w(\theta_i) \le \Gamma_{i+1}
\bigl\langle\nabla w(\theta_{i}),\bar{H}(\theta_{i},X_{i+1})
\bigr\rangle +\Gamma_{i+1}^2\frac{C_{w}}{2}\bigl|H(
\theta_{i},X_{i+1})\bigr|^{2} ,
\end{eqnarray*}
so on $A_m$ we may bound
\begin{eqnarray*}
w(\theta_{\sigma_m})-w(\theta_{\tau_m}) &\le& \sum
_{i=\tau_m}^{\sigma_m-1} \Gamma_{i+1} \bigl\langle\nabla
w(\theta_{i}),\bar{H}(\theta_{i},X_{i+1}) \bigr
\rangle +\Gamma_{i+1}^2\frac{C_{w}}{2}\bigl|H(
\theta_{i},X_{i+1})\bigr|^{2}
\\
&\le&2\sup_{k>m}\Biggl\llvert \sum_{i=m}^{k}
\Gamma_{i+1} \bigl\langle\nabla w(\theta_{i}), \bar{H}(
\theta_{i},X_{i+1}) \bigr\rangle\Biggr\rrvert +\sum
_{i=m}^\infty\Gamma_{i+1}^2
\frac{C_{w}}{2}\bigl|H(\theta_{i},X_{i+1})\bigr|^{2}\\
&\hphantom{:}=:&C_m
\end{eqnarray*}
by a similar argument as in \eqref{eq:partial-sums}.
On $A_m$ one clearly has
$w(\theta_{\sigma_m})-w(\theta_{\tau_m-1})>\delta$, implying that
$C_m + w(\theta_{\tau_m})-w(\theta_{\tau_m-1})>\delta$.
We deduce that
\[
\tilde{A}_m := \Bigl\{C_m + \sup_{i\ge m}
\bigl[w(\theta_{i+1})-w(\theta_i)\bigr] > \delta \Bigr\}
\supset A_m .
\]
The sets $\tilde{A}_m$ are clearly decreasing with respect to $m$ and
$\lim_{m\to\infty}
\P_{\theta,x} (\tilde{A}_m )=0$
by Lemma~\ref{lem:D_n_has_prob_1} and
because Condition \ref{cond:abstract-growth}\ref{eq:square-summable} and \eqref{eq:noise-negligible} imply
$\lim_{m\to\infty} C_m = 0$.
This concludes the proof, because $\bigcup_{m\ge n} A_m \subset
\bigcup_{m\ge n} \tilde{A}_m = \tilde{A}_n$.
\end{pf}
%
%}}}

%}}}

%}}}

%%%%%%%%%%%%%%%%%%%%%%%%%%%%%%%%%%%%%%%%%%%%%%%%%%%%%%%%%%%%%%%%%%%%%%%%%%%%%%%%%
%s3 #&#
\section{Verifying noise conditions}\label{sec:noise} %{{{

%{{{
The aim of this section is to provide verifiable conditions which will
imply the conditions of the stability theorems in Section~\ref{sec:stability}.
We proceed progressively and start by a general result in
Theorem~\ref{thm:general-noise-theorem} which
ensures both Condition \ref{cond:abstract-growth} and that in
\eqref{eq:noise-negligible} hold given a set of abstract conditions
involving some expectations as well as
properties of the solutions of the Poisson equation.

Condition \ref{cond:general-noise-theorem-cond}, required in Theorem~\ref{thm:general-noise-theorem},
shall be verified in detail below for a family of geometrically
ergodic Markov kernels. In Section~\ref{sec:geometric}, we first
gather general known results related to
Condition \ref{cond:general-noise-theorem-cond}\ref{eq:V-exp-bound}
and \ref{eq:poisson-bound}. In Section~\ref{sec:continuous-kernels},
we consider the case where the mapping $\theta\to P_\theta$ is
H\"{o}lder continuous,
which allows us to establish Condition \ref
{cond:general-noise-theorem-cond}\ref{eq:contin-sum}.
In Section~\ref{sec:random-step-size}, we consider the
case where the aforementioned H\"{o}lder continuity may not hold, and
a continuity is enforced by using a random step
size sequence, allowing us to recover
Condition \ref{cond:general-noise-theorem-cond}\ref{eq:contin-sum} in such situations.
%%%%%%%%%%%%%%%%%%%%
%
%co3.1 #&#
\begin{condition}\label{cond:general-noise-theorem-cond} %{{{
Condition \ref{cond:w} holds with constants
$(\xi_i)_{i\ge0}$ and $\alpha_w\in(0,\infty)$.
For all $\theta\in\hat{\Theta}$, the solution $g_\theta\dvtx \mathsf
{X}\to\Theta$ to the
Poisson equation
$g_\theta(x)- P_\theta g_\theta(x) \equiv\bar{H}(\theta,x)$ exists and
for all $i\ge0$
the step size $\Gamma_{i+1}$ is independent
of $\F_i$ and $X_{i+1}$.
Moreover, there exist a measurable function $V\dvtx \mathsf{X}\to[1,\infty
)$ and
constants $c<\infty$, $\beta_H,\beta_g\in[0,1/2]$ and
$\alpha_g,\alpha_H,\alpha_V\in[0,\infty)$ such that
for all $(\theta,x)\in\mathcal{R}_0\times\mathsf{X}$
\renewcommand\theenumi{(\roman{enumi})}
\renewcommand\labelenumi{\theenumi}
\begin{enumerate}[(vii)]
\item\label{eq:H-bound}\hspace*{4pt} $\displaystyle \sup_{\theta\in\mathcal{R}_i} \bigl|H(\theta,x)\bigr| \le c \xi_i^{\alpha_H}
V^{\beta_H}(x)$,\vspace*{5pt}

\item\label{eq:V-exp-bound} \hspace*{4pt}$\E_{\theta,x}\bigl[V(X_i)\bigr]  \le c \xi_i^{\alpha_V}
V(x)$,\vspace*{5pt}

\item\label{eq:poisson-bound} \hspace*{4pt}$\displaystyle \sup_{\theta\in\mathcal{R}_i} \bigl[ \bigl|g_\theta(x)\bigr| + \bigl|P_\theta
g_\theta(x)\bigr| \bigr] \le c \xi_i^{\alpha_g}
V^{\beta_g}(x)$,\vspace*{3pt}

\item\label{eq:contin-sum} \hspace*{4pt}$\displaystyle \sum_{i=1}^\infty\E[\Gamma_{i+1}]
\xi_i^{\alpha_w} \E_{\theta,x} \bigl[\bigl|P_{\theta_i}g_{\theta_i}(X_i)
-P_{\theta_{i-1}}g_{\theta_{i-1}}(X_i)\bigr| \bigr] <\infty $,\vspace*{4pt}

\item\label{eq:martingale-suff}\hspace*{4pt} $\displaystyle \sum_{i=1}^\infty\E\bigl[
\Gamma_i^2\bigr] \xi_i^{2\alpha_w+2((\alpha_H +
\beta_H
\alpha_V)\vee( \alpha_g + \beta_g \alpha_V))} <
\infty $,\vspace*{4pt}

\item\label{eq:prod-consecutive-weights}\hspace*{4pt} $\displaystyle \sum_{i=1}^\infty \E[\Gamma_{i+1}
\Gamma_i] \xi_{i}^{\alpha_H+\alpha_g+(\beta
_H+\beta_g)\alpha_V} <\infty $,\vspace*{3pt}

\item\label{eq:diff-consecutive-weights} \hspace*{4pt}$\displaystyle \sum_{i=1}^\infty\bigl|\E[\Gamma_{i+1}-
\Gamma_i]\bigr| \xi_{i}^{\alpha_w+\alpha_g+\beta_g\alpha_V} < \infty$,
\end{enumerate}
where we write
$\E:=\E_{\theta,x}$ whenever the expectation does not depend on
$\theta$ and $x$.
\end{condition}
%
%}}}
%%%%%%%%%%%%%%%%%%%%
%
%re3.2 #&#
\begin{remark}\label{rem:discussion-of-constants} %{{{
These assumptions call for various comments of practical relevance to
the actual implementation
of the algorithm with expanding projections. Once $H(\cdot,\cdot)$
and $\{P_\theta\}_{\theta\in\Theta}$ are chosen the
user is left with the choice of $(\xi_i)_{i \geq0}$ and $(\Gamma_i)_{i \geq0}$, which must in particular
satisfy the summability conditions above. For the purpose of efficiency
we would like $(\xi_i)_{i \geq0}$ to grow as fast as possible, as we
may otherwise slow convergence down.
A common choice for the step-size sequence is $\Gamma_i = ci^{-\eta}$
for some constants
$c\in(0,\infty)$ and $\eta\in(1/2,1]$ -- this implies a required
condition to establish convergence. The sequence $(\xi_i)_{i \geq0}$
is determined
by the user through the choice of the sequence of reprojection sets
$(\mathcal{R}_i)_{i \geq0}$ and we point out that the constants
$\alpha_H,\alpha_V$ and $\alpha_g$ typically depend on that choice
(whereas $\beta_H$ and $\beta_g$ typically do not). We show how these
constants can be obtained from the properties of $\{P_\theta\}_{\theta
\in\Theta}$ in Sections \ref{sec:geometric}--\ref{sec:random-step-size}.
Now if $(\xi_i)_{i \geq0}$ is increasing
at a rate slower than any power sequence, for example of the order
$\log i$ or
$i^{(\log i)^{-p}}$ for some $p\in(0,1)$, then it is easy to see that
the summability conditions
\ref{eq:martingale-suff}--\ref{eq:diff-consecutive-weights} are
always satisfied. In the situation where $\xi_i = i^p$ for $p \in
(0,1]$, then the conditions
\ref{eq:martingale-suff}--\ref{eq:diff-consecutive-weights}
require stricter assumptions on $\eta$ and the constants $\alpha_H,\alpha_V,\alpha_g,\beta_H$ and $\beta_g$ which
may not be satisfiable.
We however point out a possible sub-optimality of the results stated
above. Indeed, in order to simplify presentation we
have decided to quantify the growth of the various quantities involved
in the algorithm in terms of powers of $(\xi_i)_{i \geq0}$ only, whereas
other scales may be possible, such as $\log(\xi_i)$, in which case
some of the constants $\alpha_H,\alpha_V$ or $\alpha_g$ may be taken
arbitrarily
small in the statement above. It is also possible to revisit our proofs
with such more precise estimates and obtain a set of weaker assumptions.

In practice, the conditions \ref{eq:poisson-bound}
and \ref{eq:contin-sum} add more requirements which are
inter-related with
\ref{eq:martingale-suff}--\ref{eq:diff-consecutive-weights};
Propositions \ref{prop:final-prop-cont} and \ref{prop:nonsmooth-noise}
summarise the conditions when $\theta\mapsto P_\theta$ admits a
H\"{o}lder-continuity, and when a random step size sequence is used to
satisfy \ref{eq:contin-sum}, respectively.
Appendix~\ref{sec:nom} contains a summary of the related constants.
\end{remark}
%
%}}}
%%%%%%%%%%%%%%%%%%%%
%
%th3.3 #&#
\begin{theorem}\label{thm:general-noise-theorem} %{{{
Suppose Conditions \ref{cond:w}
and \ref{cond:general-noise-theorem-cond} hold
and for all $i\ge0$
the projections satisfy
$|\theta_{i+1}-\theta_i|\le|\theta_{i+1}^*-\theta_i|$.
Then, for all $(\theta,x)\in\mathcal{R}_0\times\mathsf{X}$,
%
%e22 #&#
%e23 #&#
\begin{eqnarray}
\label{eq:square-summable-noise}\E_{\theta,x} \Biggl[{\sum_{i=0}^\infty}
\Gamma_{i+1}^{2}\xi_{i}^{2\alpha_w}\bigl|H(
\theta_{i},X_{i+1})\bigr|^{2} \Biggr] &<&\infty,
\\
\label{eq:noise-negligible-poisson}\lim_{m\to\infty} \E_{\theta,x} \Biggl[ \sup_{n\ge m} \Biggl
\llvert \sum_{i=m}^{n}\Gamma_{i+1}
\bigl\langle\nabla w(\theta_{i}),\bar{H}(\theta_{i},X_{i+1})
\bigr\rangle\Biggr\rrvert \Biggr] &= &0 .
\end{eqnarray}
\end{theorem}
%
%}}}
%
\begin{pf} %{{{
Throughout the proof, $C$ denotes a constant which may have a
different value upon each appearance.
For \eqref{eq:square-summable-noise}, we may use
Condition \ref{cond:general-noise-theorem-cond}\ref{eq:H-bound} and \ref{eq:V-exp-bound}
with Jensen's inequality to obtain
\begin{eqnarray*}
\E_{\theta,x} \Biggl[{\sum_{i=0}^\infty}
\Gamma_{i+1}^{2}\xi_{i}^{2\alpha_w}\bigl|H(
\theta_{i},X_{i+1})\bigr|^{2} \Biggr] &\le& C {\sum
_{i=0}^\infty} \E\bigl[\Gamma_{i+1}^{2}
\bigr] \xi_i^{2\alpha_w+2\alpha_H} \E_{\theta,x}\bigl[V^{2\beta_H}(X_{i+1})
\bigr]
\\
&\le& C V^{2\beta_H}(x){\sum_{i=0}^\infty}
\E\bigl[\Gamma_{i+1}^{2}\bigr] \xi_i^{2\alpha_w+2\alpha_H+2\beta_H\alpha
_V}
,
\end{eqnarray*}
where the sum converges by
Condition
\ref{cond:general-noise-theorem-cond}\ref{eq:martingale-suff}.

Consider then \eqref{eq:noise-negligible-poisson}, and
denote the partial sums for $n\ge m\ge1$ as
\[
A_{m,n} := \sum_{i=m}^{n}
\Gamma_{i+1} \bigl\langle\nabla w(\theta_{i}),\bar{H}(
\theta_{i},X_{i+1}) \bigr\rangle.
\]
Since $\bar{H}(\theta_i,X_{i+1}) = g_{\theta_i}(X_{i+1})-P_{\theta_i}
g_{\theta_i}(X_{i+1})$, we may write
\begin{eqnarray*}
&&\Gamma_{i+1} \bigl\langle\nabla w(\theta_{i}),\bar{H}(
\theta_{i},X_{i+1}) \bigr\rangle \\
&&\quad  = \Gamma_{i+1}
\bigl\langle\nabla w(\theta_{i}), g_{\theta_{i}}(X_{i+1})
-P_{\theta_{i}}g_{\theta_{i}}(X_{i}) \bigr\rangle
\\
&&\qquad {} +\Gamma_{i+1} \bigl\langle\nabla w(\theta_{i}),
P_{\theta_{i}}g_{\theta_{i}}(X_{i}) -P_{\theta_{i-1}}g_{\theta_{i-1}}(X_{i})
\bigr\rangle
\\
&&\qquad {} +\Gamma_{i+1} \bigl\langle\nabla w(\theta_{i}),
P_{\theta_{i-1}}g_{\theta_{i-1}}(X_{i}) -P_{\theta_{i}}g_{\theta_{i}}(X_{i+1})
\bigr\rangle,
\end{eqnarray*}
where the last term can be written as
\begin{eqnarray*}
&&\Gamma_{i+1} \bigl\langle\nabla w(\theta_{i}),
P_{\theta_{i-1}}g_{\theta_{i-1}}(X_{i}) -P_{\theta_{i}}g_{\theta_{i}}(X_{i+1})
\bigr\rangle
\\
&&\quad =\Gamma_{i+1} \bigl\langle\nabla w(\theta_{i})-\nabla w(
\theta_{i-1}), P_{\theta_{i-1}}g_{\theta_{i-1}}(X_{i})
\bigr\rangle
\\
&&\qquad {}+\Gamma_{i} \bigl\langle\nabla w(\theta_{i-1}),
P_{\theta_{i-1}}g_{\theta_{i-1}}(X_{i}) \bigr\rangle -
\Gamma_{i+1} \bigl\langle\nabla w(\theta_{i}),
P_{\theta_{i}}g_{\theta_{i}}(X_{i+1}) \bigr\rangle
\\
&&\qquad {}+(\Gamma_{i+1}-\Gamma_i) \bigl\langle\nabla w(
\theta_{i-1}), P_{\theta_{i-1}}g_{\theta_{i-1}}(X_{i})
\bigr\rangle.
\end{eqnarray*}
When summing up, the middle term on the right is telescoping, so in total
we may write $A_{m,n} = \sum_{k=1}^5 R_{m,n}^{k}$ where
\begin{eqnarray*}
R_{m,n}^{1} & :=&\sum_{i=m}^{n}
\Gamma_{i+1} \bigl\langle\nabla w(\theta_{i}),
g_{\theta_{i}}(X_{i+1})-P_{\theta_{i}}g_{\theta_{i}}(X_{i})
\bigr\rangle,
\\
R_{m,n}^{2} & :=&\sum_{i=m}^{n}
\Gamma_{i+1} \bigl\langle\nabla w(\theta_{i}),
P_{\theta_{i}}g_{\theta_{i}}(X_{i}) -P_{\theta_{i-1}}g_{\theta_{i-1}}(X_{i})
\bigr\rangle,
\\
R_{m,n}^{3} & :=&\sum_{i=m}^{n}
\Gamma_{i+1} \bigl\langle\nabla w(\theta_{i})-\nabla w(
\theta_{i-1}), P_{\theta_{i-1}}g_{\theta_{i-1}}(X_{i})
\bigr\rangle,
\\
R_{m,n}^{4} & :=& \Gamma_{m} \bigl\langle\nabla
w(\theta_{m-1}), P_{\theta_{m-1}}g_{\theta_{m-1}}(X_{m})
\bigr\rangle - \Gamma_{n+1} \bigl\langle\nabla w(\theta_{n}),
P_{\theta_{n}}g_{\theta_{n}}(X_{n+1}) \bigr\rangle,
\\
R_{m,n}^{5} & :=&\sum_{i=m}^{n}(
\Gamma_{i+1}-\Gamma_{i}) \bigl\langle\nabla w(
\theta_{i-1}), P_{\theta_{i-1}}g_{\theta_{i-1}}(X_{i})
\bigr\rangle.
\end{eqnarray*}
We shall show that \eqref{eq:noise-negligible-poisson} holds for
each of these five terms in turn, which is sufficient to yield the claim.

Notice that $\{R_{m,i}^1\}_{i=m}^n$ is a martingale with respect to the
filtration
$\{\F_i\}_{i=m}^n$, whence
\begin{eqnarray*}
\E_{\theta,x} \bigl[\bigl|R_{m,n}^{1}\bigr|^{2}
\bigr] & =&\sum_{i=m}^{n} \E_{\theta,x}
\bigl[\Gamma_{i+1}^{2} \bigl| \bigl\langle\nabla w(
\theta_{i}), g_{\theta_{i}}(X_{i+1}) -P_{\theta_{i}}g_{\theta_{i}}(X_{i})
\bigr\rangle\bigr|^{2} \bigr]
\\
& \leq& C \sum_{i=m}^{n}
\xi_{i}^{2\alpha_w}\E \bigl[ \Gamma_{i+1}^{2}
\bigr] \E_{\theta,x} \bigl[ \bigl|g_{\theta_{i}}(X_{i+1})\bigr|^{2}
+\bigl|P_{\theta_{i}}g_{\theta_{i}}(X_{i})\bigr|^{2} \bigr]
\\
& \leq& C \sum_{i=m}^{n}
\xi_{i}^{2\alpha_w+2\alpha_g} \E \bigl[\Gamma_{i+1}^{2}
\bigr] \E_{\theta,x} \bigl[V^{2\beta_g}(X_{i+1})
+V^{2\beta_g}(X_{i}) \bigr]
\\
& \leq& CV^{2\beta_g}(x) \sum_{i=m}^{n}
\xi_{i+1}^{2\alpha_w+2\alpha_g+2\beta_g\alpha_V}\E\bigl[\Gamma_{i+1}^2
\bigr] ,
\end{eqnarray*}
by the fact that $\Gamma_{i+1}$ is independent of $\F_i$ and $X_{i+1}$,
Condition \ref{cond:w}\ref{item:xi},
Condition \ref{cond:general-noise-theorem-cond}\ref{eq:V-exp-bound} and
\ref{eq:poisson-bound}. Now,
Jensen's and Doob's inequality imply
\[
\Bigl(\E_{\theta,x} \Bigl[\sup_{n\ge m}\bigl|R_{m,n}^{1}\bigr|
\Bigr] \Bigr)^2 \le\E_{\theta,x} \Bigl[\sup_{n\ge m}\bigl|R_{m,n}^{1}\bigr|^{2}
\Bigr] \leq CV^{2\beta_g}(x)\sum_{i=m}^{\infty}
\xi_{i+1}^{2\alpha_w+2\alpha_g+2\beta_g\alpha_V}\E\bigl[\Gamma_{i+1}^2
\bigr] .
\]
This yields
$\lim_{m\rightarrow\infty}
\E_{\theta,x} [\sup_{n\ge m}|R_{m,n}^{1}| ] = 0$,
because the term on the right
tends to zero as $m\to\infty$
by Condition \ref{cond:general-noise-theorem-cond}\ref{eq:martingale-suff}.

For the second term $R_{m,n}^{2}$, we may simply write
\begin{eqnarray*}
\E_{\theta,x} \Bigl[ \sup_{n\geq m}\bigl|R_{m,n}^{2}\bigr|
\Bigr] &\le&\E_{\theta,x} \Biggl[ { \sum_{i=m}^{\infty}}
\bigl|\Gamma_{i+1} \bigl\langle\nabla w(\theta_{i}),
P_{\theta_{i}}g_{\theta_{i}}(X_{i})-P_{\theta_{i-1}}g_{\theta_{i-1}}
(X_{i}) \bigr\rangle \bigr| \Biggr]
\\
&\le& C\sum_{i=m}^{\infty}
\xi_i^{\alpha_w} \E[\Gamma_{i+1}] \E_{\theta,x}
\bigl[\bigl| P_{\theta_{i}}g_{\theta_{i}}(X_{i}) -P_{\theta_{i-1}}g_{\theta_{i-1}}(X_{i})
\bigr| \bigr] ,
\end{eqnarray*}
which converges to zero as $m\to\infty$ by
Condition \ref{cond:general-noise-theorem-cond}\ref{eq:contin-sum}.

Now we inspect $R_{m,n}^{3}$.
First, since the Hessian is bounded as in
Condition \ref{cond:w}\ref{item:w-hessian}, we have
\begin{eqnarray*}
\bigl|\nabla w(\theta_{i})-\nabla w(\theta_{i-1})\bigr| &\leq&
C_{w}|\theta_{i}-\theta_{i-1}| \le
C_w\bigl|\theta_{i}^*-\theta_{i-1}\bigr| =
C_w \Gamma_{i} \bigl|H(\theta_{i-1},X_{i})\bigr|
\\
&\le& C_w \xi_{i}^{\alpha_H} \Gamma_{i}
V^{\beta_H}(X_{i}) ,
\end{eqnarray*}
and consequently
\begin{eqnarray*}
\E_{\theta,x} \Bigl[\sup_{n\ge m} \bigl| R_{m,n}^3
\bigr| \Bigr] &\le& C \sum_{i=m}^\infty \E[
\Gamma_{i+1}\Gamma_i] \xi_{i}^{\alpha_g+\alpha_H}
\E_{\theta,x}\bigl[V^{\beta_g+\beta_H}(X_i)\bigr]
\\
&\le& C V^{\beta_g+\beta_H}(x)\sum_{i=m}^\infty
\E[\Gamma_{i+1}\Gamma_i] \xi_{i}^{\alpha_g+\alpha_H+(\beta_g+\beta_H)\alpha_V}
,
\end{eqnarray*}
by Condition \ref{cond:general-noise-theorem-cond}\ref{eq:H-bound}, \ref{eq:V-exp-bound} and \ref{eq:poisson-bound}.
The claim follows for $R_{m,n}^3$ by
Condition \ref{cond:general-noise-theorem-cond}\ref{eq:prod-consecutive-weights}.

Let us then focus on $R_{m,n}^{4}$.
We have for any $i\geq m$
\begin{eqnarray*}
\bigl\llvert \Gamma_{i} \bigl\langle\nabla w(\theta_{i-1}),
P_{\theta_{i-1}}g_{\theta_{i-1}}(X_{i}) \bigr\rangle\bigr\rrvert &
\leq C\Gamma_{i}\xi_{i}^{\alpha_w+\alpha_g} V^{\beta_g}(X_i)
.
\end{eqnarray*}
Now we have
\begin{eqnarray*}
\E_{\theta,x} \Bigl[ \sup_{n\geq m}\bigl|R_{m,n}^{4}\bigr|^{2}
\Bigr] & \leq& C \sum_{i=m}^\infty
\xi_{i}^{2\alpha_w+2\alpha_g} \E\bigl[\Gamma_{i}^2
\bigr] \E_{\theta,x} \bigl[ V^{2\beta_g}(X_i) \bigr]
\\
&\le& C V^{2\beta_g}(x) \sum_{i=m}^\infty
\xi_{i}^{2\alpha_w+2\alpha
_g+2\beta_g\alpha_V} \E\bigl[\Gamma_{i}^2
\bigr] ,
\end{eqnarray*}
so \eqref{eq:noise-negligible-poisson} holds for $R_{m,n}^4$ by
Condition \ref{cond:general-noise-theorem-cond}\ref{eq:martingale-suff}.

We shall apply
Lemma~\ref{lemma:random-weights} below for the
last term $R_{m,n}^{5}$, with $Z_i:=\Gamma_i$
and
\[
B_{i-1}:= \bigl\langle\nabla w(\theta_{i-1}),
P_{\theta_{i-1}}g_{\theta_{i-1}}(X_{i}) \bigr\rangle \qquad \mbox{with }
|B_{i-1}| \le C \xi_{i-1}^{\alpha_w+\alpha_g}V^{\beta_g}(X_{i})
.
\]
By the independence of $\Gamma_{i+1}$ and $\Gamma_i$,
and because $\xi_{i+1}\ge\xi_i\ge\xi_{i-1}$,
we easily establish the required bounds
\begin{eqnarray*}
\sum_{i=1}^\infty \Var(
\Gamma_{i+1}-\Gamma_i)\E_{\theta,x}
\bigl[B_{i-1}^2\bigr] & \le& C V^{2\beta_g}(x) \sum
_{i=1}^\infty \E\bigl[\Gamma_{i}^2
\bigr]\xi_{i}^{2\alpha_w+2\alpha_g+2\beta_g\alpha_V} <\infty,
\\
\sum_{i=1}^\infty\bigl|\E[\Gamma_{i+1}-
\Gamma_i]\bigr| \E [|B_{i-1}| ] &\le& C V^{\beta_g}(x) \sum
_{i=1}^\infty\bigl|\E[\Gamma_{i+1}-
\Gamma_i]\bigr| \xi_{i}^{\alpha_w+\alpha_g+\beta_g\alpha_V} < \infty,
\end{eqnarray*}
by Condition \ref{cond:general-noise-theorem-cond}\ref{eq:martingale-suff} and
\ref{eq:diff-consecutive-weights}, respectively.
\end{pf}
%
%}}}
%%%%%%%%%%%%%%%%%%%%
%
%le3.4 #&#
\begin{lemma}\label{lemma:random-weights} %{{{
Let $\{\G_i\}_{i\ge0}$ be a filtration and for all $i\ge0$ let
$B_i$ and $Z_i$ be $\G_i$-adapted random variables
so that $Z_i$ is independent of $\G_{i-1}$ and
\[
\sum_{i=1}^\infty\Var(Z_{i+1}-Z_i)
\E\bigl[B_{i-1}^2\bigr] <\infty\quad \mbox{and}\quad  \sum
_{i=1}^\infty\bigl|\E[Z_{i+1}-Z_i]\bigr|
\E [|B_{i-1}| ] <\infty.
\]
Then,
\[
\lim_{m\to\infty} \E \Biggl[\sup_{n\ge m} \Biggl|\sum
_{i=m}^n (Z_{i+1}-Z_i)B_{i-1}
\Biggr| \Biggr] = 0 .
\]
\end{lemma}
%
%}}}
%
\begin{pf} %{{{
Suppose for now that
$m$ is even and $n$ odd and denote $m=2\bar{m}$ and $n=2\bar{n}+1$.
Write the sum
%
%e24 #&#
\begin{equation}\label{eq:odd-even}
\sum_{i=m}^{n} (Z_{i+1}-Z_i)B_{i-1}
= \sum_{j=\bar{m}}^{\bar{n}} (Z_{2j+1}-Z_{2j})B_{2j-1}
+ \sum_{k=\bar{m}}^{\bar{n}} (Z_{2k+2}-Z_{2k+1})B_{2k}
.
\end{equation}
We shall first show that the claim holds for the first term on the right.
Denote $\bar{\G}_j = \G_{2j+1}$, $\bar{Z}_j=Z_{2j+1}-Z_{2j}$
and $\bar{B}_{j-1}=B_{2j-1}$. Observe that
$\E[\bar{Z}_j\mid\bar{\G}_{j-1}] = \E[\bar{Z}_j]$
and write
\[
\sum_{j=\bar{m}}^{\bar{n}} (Z_{2j+1}-Z_{2j})B_{2j-1}
= \sum_{j=\bar{m}}^{\bar{n}} \bigl(
\bar{Z}_{j}-\E[\bar{Z}_j] \bigr) \bar{B}_{j-1}
+ \sum_{j=\bar{m}}^{\bar{n}} \E[\bar{Z}_j]
\bar{B}_{j-1} .
\]
Now, the first term on the right-hand side is a martingale with respect to
$\bar{\G}_j$, and so by Doob's inequality and by assumption
\begin{eqnarray*}
\E \Biggl[\sup_{\bar{n}\ge\bar{m}} \Biggl(\sum_{j=\bar{m}}^{\bar{n}}
\bigl(\bar{Z}_{j}-\E[\bar{Z}_j] \bigr)
\bar{B}_{j-1} \Biggr)^2 \Biggr] \le4\sum
_{j=\bar{m}}^{\infty} \Var(\bar{Z}_j)\E\bigl[
\bar{B}_{j-1}^2\bigr] \mathop{\xrightarrow}^{\bar{m}\to\infty} 0 .
\end{eqnarray*}
For the second term, by assumption
\[
\E \Biggl[\sup_{\bar{n}\ge\bar{m}}\Biggl |\sum_{j=\bar{m}}^{\bar{n}}
\E[\bar{Z}_j]\bar{B}_{j-1} \Biggr| \Biggr] \le\sum
_{j=\bar{m}}^{\infty} \bigl|\E[\bar{Z}_j]\bigr|\E \bigl[|
\bar{B}_{j-1}| \bigr] \mathop{\xrightarrow}^{\bar{m}\to\infty} 0 .
\]
The same arguments apply also
for the second term on the right-hand side of
\eqref{eq:odd-even}, and for any integers $m\ge n\ge1$,
by a change of the indices.
\end{pf}
%
%}}}

%}}}

%%%%%%%%%%%%%%%%%%%%%%%%%%%%%%%%%%%%%%%%
%s3.1 #&#
\subsection{Geometrically ergodic Markov kernels}\label{sec:geometric} %{{{

In this section, we focus on the scenario where for any
$\theta\in\Theta$ the kernel $P_\theta$ is geometrically ergodic.
This condition is satisfied by numerous Markov chains of practical
interest, see for example,
\cite{mengersen-tweedie,fort-moulines-roberts-rosenthal,jarner-hansen}
and references therein. This section gathers together standard results
about the regularity of the solutions to the Poisson equation (see,
e.g., \cite{andrieu-moulines,sa-verifiable}).

Throughout this section, suppose $V\dvtx \mathsf{X}\to[1,\infty)$ is a
fixed measurable function. We shall denote
the $V$-norm of a measurable function $f\dvtx \mathsf{X}\to\R^d$ by $\|f\|_V
:=\sup_x |f(x)|/V(x)$. We also assume that for each
$\theta\in\hat{\Theta}$, the Markov kernel $P_\theta$ admits a
unique invariant probability measure $\pi_\theta$.
%%%%%%%%%%%%%%%%%%%%
%
%co3.5 #&#
\begin{condition}\label{cond:geom-erg} %{{{
For any $r\in(0,1]$ and any
$\theta\in\hat{\Theta}$, there exist
constants
$M_{\theta,r}\in[0,\infty)$ and $\rho_{\theta,r}\in(0,1)$,
such that for any function $\|f\|_{V^r}<\infty$
\[
\bigl| P_{\theta}^{k}(x,f)-\pi_\theta(f)\bigr| \leq
V^{r}(x) \|f\|_{V^r} M_{\theta,r} \rho_{\theta,r}^{k}
\]
for all $k\ge0$ and all $x\in\mathsf{X}$.
\end{condition}
%
%}}}
Having Condition \ref{cond:geom-erg} one can bound the $V^r$-norm of
the solutions of the Poisson equation, making the dependence on
$\theta$ explicit. This result is a restatement of \cite{andrieu-moulines}, Proposition~3, in quantitative form; we provide it here for
the reader's convenience.
%%%%%%%%%%%%%%%%%%%%
%
%pr3.6 #&#
\begin{proposition}\label{prop:poisson-bound} %{{{
Assume Condition \ref{cond:geom-erg} holds.
Then, for any
function $\|f\|_{V^r}<\infty$, the functions
$g_\theta\dvtx \mathsf{X}\to\R^d$ defined for all $\theta\in\hat
{\Theta}$
by
\[
g_\theta(x) :={\sum_{k=0}^\infty}
\bigl[P_\theta^k f(x) - \pi_\theta(f) \bigr]
\]
exist, solve the Poisson equation $g_\theta(x)-P_\theta
g_\theta(x) \equiv f(x) - \pi_\theta(f)$, and
satisfy the bound
%
%e25 #&#
\begin{equation}\label{eq:poisson-bound-geom}
\|g_\theta\|_{V^r}\vee\|P_\theta g_\theta
\|_{V^r} \le M_{\theta,r} (1-\rho_{\theta,r})^{-1}
\|f\|_{V^r} .
\end{equation}
\end{proposition}
%
%}}}
%
\begin{pf} %{{{
It is evident that $g_\theta$ solves the Poisson equation whenever the
sum converges. By the definition of $g_\theta$ and Condition
\ref{cond:geom-erg}, we have
\[
\|g_\theta\|_{V^r} \le\sum_{k=0}^\infty
\bigl\| P_\theta^k f - \pi_\theta(f)\bigr\|_{V^r}
\le M_{\theta,r} \|f\|_{V^r} \sum_{k=0}^\infty
\rho_{\theta,r}^k = M_{\theta,r} (1-\rho_{\theta,r})^{-1}
\|f\|_{V^r} .
\]
The same bound applies clearly also for $P_\theta g_\theta$,
establishing \eqref{eq:poisson-bound-geom}.
\end{pf}
%
%}}}

We also need the following simple lemma in order to establish
Condition \ref{cond:general-noise-theorem-cond}\ref{eq:V-exp-bound}.
%%%%%%%%%%%%%%%%%%%%
%
%le3.7 #&#
\begin{lemma}\label{lemma:drift} %{{{
Suppose that for all $i\ge0$ there exist constants
$\lambda_i\in[0,1)$ and $b_i\in[0,\infty)$ such that
%
%e26 #&#
\begin{equation}\label{eq:drift-bound}
\sup_{\theta\in\mathcal{R}_i} P_{\theta} V(x) \le\lambda_i V(x) +
b_i \qquad \mbox{for all } x\in\mathsf{X},
\end{equation}
and that both $(\lambda_i)_{i\ge0}$ and $(b_i)_{i\ge0}$
are non-decreasing.
Then, for any $(\theta,x)\in\mathcal{R}_0\times\mathsf{X}$
and $i\ge0$,
the bound
$\E_{\theta,x}[ V(X_{i+1})] \le(1-\lambda_{i})^{-1} (b_{i} \vee
V(x) )$
holds.
\end{lemma}
%
%}}}
%
\begin{pf} %{{{
By construction, for all $i\ge1$ we have
$\E_{\theta,x}[V(X_i)\mid\F_{i-1}] =
P_{\theta_{i-1}} V(X_{i-1})$ and $\theta_{i-1}\in\mathcal
{R}_{i-1}$, so we may
use \eqref{eq:drift-bound} iteratively to obtain
\[
\E_{\theta,x}\bigl[ V(X_{i+1})\bigr] \le\E_{\theta,x}\bigl[
\lambda_{i} V(X_{i}) + b_{i}\bigr] \le\cdots \le
\bigl(b_{i}\vee V(x) \bigr)\sum_{k=0}^{i}
\lambda_{i}^k \le\frac{b_i \vee V(x)}{1-\lambda_i} .
\]\upqed
\end{pf}
%
%}}}

Let us consider next a case where the ergodicity rates
in each projection set $\mathcal{R}_i$ are controlled
by the sequence $\xi_i$.
%%%%%%%%%%%%%%%%%%%%
%
%co3.8 #&#
\begin{condition}\label{cond:geom-rates} %{{{
Suppose Condition \ref{cond:geom-erg} holds with constants
$M_{\theta,r},\rho_{\theta,r}$ satisfying
\[
\sup_{\theta\in\mathcal{R}_i} M_{\theta,r} \le c_r\xi_i^{\alpha_M}
\quad \mbox{and}\quad  \sup_{\theta\in\mathcal{R}_i} (1-\rho_{\theta,r})^{-1} \le
c_r\xi_i^{\alpha_\rho}
\]
for some constants $\alpha_M,\alpha_\rho\in[0,\infty)$, and
a constant $c_r\in[0,\infty)$ depending only on $r$.
\end{condition}
%
%}}}
%%%%%%%%%%%%%%%%%%%%
%
%pr3.9 #&#
\begin{proposition}\label{prop:geom-rates} %{{{
If Condition \ref{cond:geom-rates} holds,
then
Condition \ref{cond:general-noise-theorem-cond}\textup{\ref{eq:poisson-bound}} holds with $\alpha_g = \alpha_H+
\alpha_M +\alpha_\rho$ and $\beta_g=\beta_H$.
\end{proposition}
%
%}}}
%
\begin{pf} %{{{
Corollary of Proposition~\ref{prop:poisson-bound} with $r=\beta_g$.
\end{pf}
%
%}}}

Finally, we shall state a result similar to \cite{saksman-vihola}, Lemma~3, yielding
Condition \ref{cond:geom-erg} from simultaneous, but $\theta$-dependent,
drift and minorisation conditions. These conditions can be verified
for random-walk Metropolis kernels with a target distribution having
super-exponential tail decay and
sufficiently regular tail contours
\cite{jarner-hansen,andrieu-moulines,saksman-vihola,vihola-asm}.
%%%%%%%%%%%%%%%%%%%%
%
%co3.10 #&#
\begin{condition}\label{cond:drift-mino} %{{{
Suppose that $P$ is an irreducible and aperiodic
Markov kernel with invariant
distribution $\pi$, that
there exists a Borel set $C\subset\mathsf{X}$,
a probability measure $\nu$ concentrated on $C$,
constants $\lambda\in[0,1)$, $b<\infty$ and
$\delta\in(0,1]$ such that
$v :=\sup_{x\in C} V(x) <\infty$ and
\begin{eqnarray*}
P V(x)&\le&\lambda V(x) + b \mathbb{I}\{x\in C\}\qquad \mbox{for all } x\in\mathsf{X},
\\
P(x,A)&\ge&\delta\nu(A)  \qquad \mbox{for all } x\in C \mbox{ and any Borel set } A\subset
\mathsf{X}.
\end{eqnarray*}
\end{condition}
%
%}}}
%%%%%%%%%%%%%%%%%%%%
%
%pr3.11 #&#
\begin{proposition}\label{prop:drift-to-erg} %{{{
Assume Condition \ref{cond:drift-mino}.
Then, for any $r\in(0,1]$ there exists
a constant $c_r^*\in[1,\infty)$ depending
only on $r$ such that
for all $\|f\|_{V^r}<\infty$ and $k\geq1$
\[
\bigl\llVert P^{k}(x,f)-\pi(f)\bigr\rrVert_{V^r} \leq
V^r(x)M_r\rho_r^{k}\|f
\|_{V^r} ,
\]
where the constants $M_r\in[1,\infty)$ and $\rho_r\in(0,1)$
are defined in terms of the constants in Condition
\ref{cond:drift-mino} as follows
\begin{eqnarray*}
\rho_r &:=&1 - \bigl[c_r^* (1-\lambda)^{-4}
\delta^{-13} \bar{b}^{6} \bigr]^{-1},
\\
M_r &:=&c_r^*(1-\lambda)^{-4}
\delta^{-15}\bar{b}^{7} ,
\end{eqnarray*}
where $\bar{b}:=b\vee v\ge1$.
\end{proposition}
%
%}}}
The proof of Proposition~\ref{prop:drift-to-erg}
is given in Appendix~\ref{sec:drift}.

%}}}

%%%%%%%%%%%%%%%%%%%%%%%%%%%%%%%%%%%%%%%%
%s3.2 #&#
\subsection{Smooth family of Markov kernels}\label{sec:continuous-kernels} %{{{

In many practically interesting settings, the mapping $\theta\mapsto
P_\theta$, possibly restricted to a suitable set, satisfies a
H\"{o}lder continuity condition.
This continuity allows one to establish Condition
\ref{cond:general-noise-theorem-cond}\ref{eq:contin-sum} in a
natural way \cite{andrieu-moulines,sa-verifiable,benveniste-metivier-priouret}.
We restate these results in a quantitative manner below, so that they
are directly applicable in the present setting. The H\"{o}lder
continuity condition
is given as follows.
%%%%%%%%%%%%%%%%%%%%
%
%co3.12 #&#
\begin{condition}\label{cond:continuity} %{{{
Suppose Condition \ref{cond:geom-erg} holds
and for any $\theta,\theta'\in\hat{\Theta}$, there exist a constant
$D_{\theta,\theta' ,r}\in[0,\infty)$
and a constant $\beta_D\in(0,\infty)$ independent of
$\theta$, $\theta'$ and $r$
such that for any
function $\|f\|_{V^r}<\infty$
\[
\llVert P_{\theta}f-P_{\theta^{\prime}}f\rrVert_{V^r} \leq\llVert
f\rrVert_{V^r}D_{\theta,\theta' ,r} |\theta-\theta^{\prime}|^{\beta_D}
.
\]
\end{condition}
%
%}}}

We consider below only the case when $P_\theta$ and $P_{\theta'}$
admit the same stationary measure;
this is a commonly encountered in
adaptive Markov chain Monte Carlo.
The general case is slightly more involved, but can be handled as well;
we refer the reader to \cite{sa-verifiable} for details.
We start by a lemma characterising the difference of the iterates of
the kernels.
%%%%%%%%%%%%%%%%%%%%
%
%le3.13 #&#
\begin{lemma}\label{lemma:kernel-iterates} %{{{
Assume Condition \ref{cond:continuity}
holds and $f$ is a measurable function with
$\|f\|_{V^r}<\infty$ and that $\pi_\theta= \pi_{\theta'}=:\pi$.
Then,
for any $k\ge0$
\begin{eqnarray*}
\bigl\| P_\theta^k f - P_{\theta'}^k f
\bigr\|_{V^r} \le M_{\theta,r} M_{\theta' ,r} D_{\theta,\theta' ,r} k (
\rho_{\theta,r}\vee\rho_{\theta' ,r})^{k-1} \bigl|\theta-
\theta'\bigr|^{\beta_D} \|f\|_{V^r} .
\end{eqnarray*}
\end{lemma}
%
%}}}
%
\begin{pf} %{{{
We use the following telescoping decomposition
\begin{eqnarray*}
P_\theta^k f - P_{\theta'}^k f = \sum
_{j=1}^k P_\theta^{k-j}(P_\theta-P_{\theta'})P_{\theta'}^{j-1}
f = \sum_{j=1}^k \bigl(P_\theta^{k-j}-
\Pi\bigr) (P_\theta-P_{\theta'}) \bigl(P_{\theta'}^{j-1}f-
\pi(f) \bigr) ,
\end{eqnarray*}
where $\Pi(x,A) :=\pi(A)$ for all
$x\in\mathsf{X}$ and all measurable $A\subset\mathsf{X}$.

By Condition \ref{cond:geom-erg} and Condition \ref{cond:continuity},
\begin{eqnarray*}
\bigl\|(P_\theta-P_{\theta'}) \bigl(P_{\theta'}^{j-1}f-
\pi(f) \bigr)\bigr\|_{V^r} &\le& \bigl\| P_{\theta'}^{j-1}f-\pi(f)
\bigr\|_{V^r} D_{\theta,\theta' ,r} \bigl|\theta-\theta^{\prime}\bigr|^{\beta_D}
\\
& \le& D_{\theta,\theta' ,r} M_{\theta' ,r} \rho_{\theta' ,r}^{j-1}
\|f\|_{V^r} \bigl|\theta-\theta^{\prime}\bigr|^{\beta_D} .
\end{eqnarray*}
Writing then
\begin{eqnarray*}
\bigl\|P_\theta^k f - P_{\theta'}^k f
\bigr\|_{V^r} &\le k \displaystyle \sup_{1\le j\le k} \bigl\| \bigl(P_\theta^{k-j}-
\Pi\bigr) (P_\theta-P_{\theta'}) \bigl(P_{\theta'}^{j-1}f-
\pi(f) \bigr) \bigr\|_{V^r} ,
\end{eqnarray*}
and applying Condition \ref{cond:geom-erg} once more
yields the claim.
\end{pf}
%
%}}}
%%%%%%%%%%%%%%%%%%%%
%
%pr3.14 #&#
\begin{proposition}\label{prop:geom-erg-continuous} %{{{
Assume Condition \ref{cond:continuity}
holds,
$\pi_\theta=\pi_{\theta'}=:\pi$
and $\|f_\theta\|_{V^r}\vee\|f_{\theta'}\|_{V^r}<\infty$.
Then, the solutions of the Poisson equation defined as
$g_\theta:=\sum_{k=0}^\infty[P_\theta^k f_\theta-
\pi_\theta(f_\theta)]$ satisfy
%
%e27 #&#
\begin{eqnarray}\label{eq:poisson-continuity-geom}
\|g_\theta-g_{\theta'}\|_{V^r} \vee \|P_\theta
g_\theta-P_{\theta'} g_{\theta'}\|_{V^r} &\le&
\frac{M_{\theta,r} M_{\theta' ,r}
D_{\theta,\theta' ,r} }{
(1-(\rho_{\theta,r} \vee\rho_{\theta' ,r}) )^{2}} |\theta-\theta'|^{\beta_D}
\|f_\theta\|_{V^r}
\nonumber
\\[-8pt]\\[-8pt]
&&{} + M_{\theta' ,r} (1-\rho_{\theta' ,r})^{-1}
\|f_\theta-f_{\theta'}\|_{V^r} .\nonumber
\end{eqnarray}
\end{proposition}
%
%}}}
%
\begin{pf} %{{{
With the estimate from Lemma~\ref{lemma:kernel-iterates},
\begin{eqnarray*}
\|g_\theta-g_{\theta'}\|_{V^r} &\le&\sum
_{k=0}^\infty \bigl( \bigl\|P_\theta^k
f_\theta- P_{\theta'}^k f_\theta
\bigr\|_{V^r} + \bigl\|P_{\theta'}^k(f_\theta-f_{\theta'})
- \pi(f_\theta- f_{\theta'})\bigr\|_{V^r} \bigr)
\\
&\le &M_{\theta,r} M_{\theta' ,r} D_{\theta,\theta' ,r}\bigl |\theta-
\theta'\bigr|^{\beta_D} \|f_\theta\|_{V^r} \sum
_{k=0}^\infty k (\rho_{\theta,r} \vee
\rho_{\theta' ,r})^{k-1}
\\
&&{}+ M_{\theta' ,r} (1-\rho_{\theta' ,r})^{-1}
\|f_\theta-f_{\theta'}\|_{V^r} .
\end{eqnarray*}
The same bound clearly holds also for
$\|P_\theta g_\theta- P_{\theta'}g_{\theta'}\|_{V^r}$
yielding \eqref{eq:poisson-continuity-geom}.
\end{pf}
%
%}}}

We shall provide some sufficient conditions to verify
Condition \ref{cond:general-noise-theorem-cond}\ref{eq:contin-sum}.
%%%%%%%%%%%%%%%%%%%%
%
%co3.15 #&#
\begin{condition}\label{cond:contin-poisson} %{{{
Condition \ref{cond:continuity} holds with
constants satisfying
$\sup_{(\theta,\theta')\in\mathcal{R}_i^2} D_{\theta,\theta' ,r}
\le
c_r^D \xi_i^{\alpha_D}$ for some constant $c_r^D\in[0,\infty)$
depending only on
$r\in(0,1]$,
Condition \ref{cond:general-noise-theorem-cond}\textup{\ref{eq:H-bound}} and \textup{\ref{eq:V-exp-bound}}
hold with constants $\alpha_H,\beta_H$ and $\alpha_V$, and
there exist constants $c<\infty$, $\alpha_{\Delta}\in[0,\infty)$
and $\beta_{\Delta}>0$
such that
\[
\sup_{(\theta,\theta')\in\mathcal{R}_i^2} \bigl\llVert H(\theta, \cdot ) -H\bigl(
\theta^{\prime}, \cdot \bigr)\bigr\rrVert_{V^{\beta_H}} \leq c
\xi_i^{\alpha_{\Delta}} \bigl|\theta-\theta^{\prime}\bigr|^{\beta
_\Delta} .
\]
\end{condition}
%
%}}}
%%%%%%%%%%%%%%%%%%%%
%
%pr3.16 #&#
\begin{proposition}\label{prop:continuity} %{{{
Suppose Conditions \ref{cond:general-noise-theorem-cond}\textup{\ref{eq:H-bound}} and
\textup{\ref{eq:V-exp-bound}},
\ref{cond:geom-rates} and
\ref{cond:contin-poisson} hold,
the constants $\beta_D,\beta_\Delta\in(0,1/\beta_H - 1]$,
for any $i\ge0$
the step size $\Gamma_i$ is
independent of $X_i$
and the projections satisfy $|\theta_{i+1}-\theta_{i}|\le
|\theta_{i+1}^*-\theta_i|$. Then,
the solutions $g_\theta$ to the Poisson equation
$g_\theta- P_\theta g_\theta= \bar{H}(\theta, \cdot )$
exist for all $\theta\in\hat{\Theta}$, and
there is a constant $c<\infty$ such that
for all $(\theta,x)\in\mathcal{R}_0\times\mathsf{X}$
\begin{eqnarray*}
&&\E_{\theta,x} \bigl|P_{\theta_i}g_{\theta_i}(X_i)
-P_{\theta_{i-1}}g_{\theta_{i-1}}(X_i)\bigr|
\\
&&\quad \le c \E\bigl[\Gamma_i^{\beta_D} \bigr] \xi_i^{2\alpha_M + 2\alpha_\rho+ \alpha_D
+ (\beta_D+1)(\beta_H\alpha_V+\alpha_H)}
V^{(\beta_D+1) \beta_H}(x)
\\
&&\qquad {} + c \E\bigl[\Gamma_i^{\beta_\Delta}\bigr] \xi_i^{\alpha_M+\alpha_\rho+\alpha_\Delta+ \beta_\Delta\alpha_H
+ (\beta_\Delta+1)\beta_H\alpha_V}
V^{(\beta_\Delta+1)\beta_H}(x) .
\end{eqnarray*}
\end{proposition}
%
%}}}
%
\begin{pf} %{{{
By assumption, both $\theta_{i}$ and $\theta_{i-1}$ are in
$\mathcal{R}_i$, so $|\theta_{i}-\theta_{i-1}|\le\Gamma_{i}
|H(\theta_{i-1},X_i)|\le c \Gamma_{i} \xi_i^{\alpha_H} V^{\beta_H}(X_i)$.
Proposition~\ref{prop:geom-erg-continuous} yields, with $r=\beta_H$ and
denoting $H_\theta(x):=H(\theta,x)$,
\begin{eqnarray*}
&&\|P_{\theta_i}g_{\theta_i} -P_{\theta_{i-1}}g_{\theta_{i-1}}
\|_{V^{\beta_H}}
\\
&&\quad \le M_{\theta_i,\beta_H} M_{\theta_{i-1},\beta_H} D_{\theta_i,\theta_{i-1},\beta_H} \bigl(1-(
\rho_{\theta_i,\beta_H} \vee \rho_{\theta_{i-1},\beta_H}) \bigr)^{-2} |
\theta_i-\theta_{i-1}|^{\beta_D} \| H_{\theta_i}
\|_{V^{\beta_H}}
\\
&&\qquad {}+ M_{\theta_{i-1},\beta_H} (1-\rho_{\theta_{i-1},\beta_H})^{-1} \|
H_{\theta_i}-H_{\theta_{i-1}}\|_{V^{\beta_H}}
\\
&&\quad \le c \xi_i^{2\alpha_M + 2\alpha_\rho+ \alpha_D } |\theta_i-
\theta_{i-1}|^{\beta_D} \| H_{\theta_i}\|_{V^{\beta_H}} + c
\xi_i^{\alpha_M + \alpha_\rho} \| H_{\theta_i}-H_{\theta_{i-1}}
\|_{V^{\beta_H}}
\\
&&\quad \le c \xi_i^{2\alpha_M + 2\alpha_\rho+ \alpha_D +
\alpha_H(1+\beta_D)} \Gamma_i^{\beta_D}
V^{\beta_D \beta_H}(X_i) + c \xi_i^{\alpha_M+\alpha_\rho+\alpha_\Delta+ \beta_\Delta
\alpha_H}
\Gamma_i^{\beta_\Delta} V^{\beta_\Delta\beta_H}(X_i) .
\end{eqnarray*}
The independence of $\Gamma_i$ and $X_i$
and
Condition \ref{cond:general-noise-theorem-cond}\ref{eq:V-exp-bound} with Jensen's inequality (we have
$(1 + (\beta_D\vee\beta_\Delta))\beta_H\in(0,1]$)
imply the claim.
\end{pf}
%
%}}}

Now, we shall consider the common case where
$(\Gamma_i)_{i\ge1}$ is a deterministic power sequence.
Then, Condition \ref{cond:general-noise-theorem-cond}
can be established.
%%%%%%%%%%%%%%%%%%%%
%
%pr3.17 #&#
\begin{proposition}\label{prop:final-prop-cont} %{{{
Suppose $\Gamma_i\equiv ci^{-\eta}$ for all $i\ge1$ with some
$c<\infty$ and $\eta\in(1/2,1]$.
Then, if the conditions of Proposition~\ref{prop:continuity} hold and
%
%e28 #&#
%e29 #&#
%e30 #&#
\begin{eqnarray}
\label{eq:poisson-cont-1}\sum_{i=1}^\infty i^{-(1+\beta_D)\eta}
\xi_i^{\alpha_w + 2\alpha_M + 2\alpha_\rho+ \alpha_D
+ (\beta_D+1)(\beta_H+\alpha_V+\alpha_H)} &<& \infty,
\\
\label{eq:poisson-cont-2}\sum_{i=1}^\infty i^{-(1+\beta_{\Delta})\eta}
\xi_i^{\alpha_M + \alpha_\rho+ \alpha_{\Delta}
+ \beta_{\Delta}\alpha_H + (\beta_{\Delta}+1)\beta_H\alpha_V}
&<&\infty,
\\
\label{eq:general-conditions-cont}\sum_{i=1}^\infty i^{-2\eta}
\xi_i^{2\alpha_w
+ 2(\alpha_H + \alpha_M + \alpha_\rho+ \beta_H + \alpha_V)} &<&\infty,
\end{eqnarray}
then, Condition
\ref{cond:general-noise-theorem-cond} holds.
\end{proposition}
%
%}}}
%
\begin{pf} %{{{
Condition \ref{cond:general-noise-theorem-cond}\ref{eq:H-bound} and \ref{eq:V-exp-bound} hold by
assumption. Propositions \ref{prop:geom-rates} and
\ref{prop:continuity} imply
Condition \ref{cond:general-noise-theorem-cond}\ref{eq:poisson-bound} with $\alpha_g = \alpha_H +
\alpha_M+\alpha_\rho$ and $\beta_g = \beta_H$.
Condition
\ref{cond:general-noise-theorem-cond}\ref{eq:contin-sum}
follows from Proposition~\ref{prop:continuity} with
\eqref{eq:poisson-cont-1} and \eqref{eq:poisson-cont-2}.

Observe then that
$\Gamma_{i+1}\Gamma_i \le\Gamma_i^2 = c^2 i^{-2\eta}$ and
by the mean value theorem
$|\Gamma_{i+1}-\Gamma_i| = c \eta(i+h_i)^{-\eta-1} \le c \eta
i^{-\eta-1}
\le\eta\Gamma_{i}^2$
where $h_i\in[0,1]$.
Conditions \ref{cond:general-noise-theorem-cond}\ref{eq:martingale-suff}--\ref{eq:diff-consecutive-weights}
follow easily from \eqref{eq:general-conditions-cont},
by the fact $\alpha_g = \alpha_H +
\alpha_M+\alpha_\rho$ and $\beta_g = \beta_H$.
\end{pf}
%
%}}}

%}}}

%%%%%%%%%%%%%%%%%%%%%%%%%%%%%%%%%%%%%%%%
%s3.3 #&#
\subsection{Non-smooth family of Markov kernels}\label{sec:random-step-size} %{{{

When the mapping $\theta\to P_\theta$ does not admit
(local) H\"{o}lder-continuity as discussed above, establishing
Condition \ref{cond:general-noise-theorem-cond} is more involved, but
possible using a random step size sequence which, in intuitive terms,
enforce continuity in a stochastic manner.
We focus on a specific step size sequence given as $\Gamma_i :=
\gamma_i \mathbb{I}\{U_i\le p_i\}$ where the $U_i$ are independent uniform
$[0,1]$ random variables and both sequences $\gamma_i$ and $p_i$ decay
to zero. It will be clear later on that these sequences must satisfy
$\sum_i \gamma_i p_i = \infty$,
$\sum_i \gamma_i^2 p_i<\infty$ and $\sum_i \gamma_i p_i^2<\infty
$; for
simplicity of exposition, we
shall consider below the particular example where $\gamma_i$ and $p_i$
decay with a power law.

The definition of $(\Gamma_i)_{i\ge1}$ above will
result in practice in keeping the value of $\theta_i$ fixed for longer and
longer (random) periods. We remark that one could consider inducing such
a behaviour also in a deterministic manner, but we do not pursue this here.
%%%%%%%%%%%%%%%%%%%%
%
%pr3.18 #&#
\begin{proposition}\label{prop:random-weights} %{{{
Assume Conditions \ref{cond:w} and \ref{cond:geom-rates} hold
and for all $i\ge1$ the step size
$\Gamma_i$ is independent of $X_i$.
Suppose also that Condition \ref{cond:general-noise-theorem-cond}\textup{\ref{eq:H-bound}} holds with $\alpha_H\in[0,\infty)$ and
$\beta_H\in[0,1/2]$, and
Condition \ref{cond:general-noise-theorem-cond}\textup{\ref{eq:V-exp-bound}} holds with
$\alpha_V\in[0,\infty)$.

Then,
the solutions $g_\theta$ to the Poisson equation
$g_\theta- P_\theta g_\theta= \bar{H}(\theta, \cdot )$
exist for all $\theta\in\hat{\Theta}$, and
there exists a constant $c<\infty$ such that
for any $(\theta,x)\in\mathcal{R}_0\times\mathsf{X}$
\[
\E_{\theta,x} \bigl[\bigl|P_{\theta_i}g_{\theta_i}(X_i)
-P_{\theta_{i-1}}g_{\theta_{i-1}}(X_i)\bigr| \bigr] \le c \P(
\Gamma_i\neq0) \xi_i^{\alpha_M+\alpha_\rho+\alpha_H+\beta_H\alpha_V} V^{\beta_H}(x)
.
\]
\end{proposition}
%
%}}}
%
\begin{pf} %{{{
The solutions $g_\theta$ to the Poisson equation exist by Proposition~\ref{prop:poisson-bound}.
If $\Gamma_i=0$ then clearly $\theta_i=\theta_{i-1}$
and so
\begin{eqnarray*}
&&\bigl|P_{\theta_i}g_{\theta_i}(X_i)
-P_{\theta_{i-1}}g_{\theta_{i-1}}(X_i)\bigr|\\
&&\quad =\mathbb{I}\{\Gamma_i\neq0\}\bigl|P_{\theta_i}g_{\theta_i}(X_i)
-P_{\theta_{i-1}}g_{\theta_{i-1}}(X_i)\bigr|
\\
&&\quad \le c \mathbb{I}\{\Gamma_i\neq 0\} \bigl(\xi_i^{\alpha_M+\alpha
_\rho}
\bigl\|H(\theta_i, \cdot )\bigr\|_{V^{\beta_H}} + \xi_{i-1}^{\alpha_M+\alpha_\rho}
\bigl\|H(\theta_{i-1}, \cdot )\bigr\|_{V^{\beta_H}} \bigr) V^{\beta_H}(X_i)
,
\end{eqnarray*}
by Proposition~\ref{prop:poisson-bound}.
The claim follows
by Conditions
\ref{cond:general-noise-theorem-cond}\ref{eq:H-bound}
and \ref{eq:V-exp-bound}, and
by the independence of $\Gamma_i$ and $X_i$.
\end{pf}
%
%}}}

Next, we shall consider the particular case where $(\Gamma_i)_{i\ge1}$
is defined by two sequences with a power decay.
%%%%%%%%%%%%%%%%%%%%
%
%pr3.19 #&#
\begin{proposition}\label{prop:nonsmooth-noise} %{{{
Let $(U_i)_{i\ge1}$ be a sequence of independent and uniformly
distributed random variables on $[0,1]$, and
assume $\Gamma_i \equiv\gamma_i \mathbb{I}\{U_i \le p_i\}$,
where the constant sequences
$(\gamma_i)_{i\ge1}\subset(0,1)$ and
$(p_i)_{i\ge1}\subset[0,1]$ are defined as
$\gamma_i :=c_\gamma i^{-\eta_\gamma}$ and
$p_i :=c_p i^{-\eta_p}$ for some $c_\gamma,c_p\in(0,\infty)$ and
$\eta_\gamma,\eta_p\in(0,1)$ such that $\eta_\gamma+\eta_p\le1$,
$2\eta_\gamma+\eta_p>1$ and $\eta_\gamma+2\eta_p>1$.

If Conditions \ref{cond:general-noise-theorem-cond}\textup{\ref{eq:H-bound}} and \textup{\ref{eq:V-exp-bound}} and
Condition \ref{cond:geom-rates} hold, and
%
%e31 #&#
%e32 #&#
\begin{eqnarray}
\label{eq:xi-random-1}\sum_{i=1}^\infty i^{-\eta_\gamma-2\eta_p}
\xi_i^{\alpha_w+\alpha_M+\alpha_\rho+\alpha_H+\beta_H\alpha_V} &<&
\infty,
\\
\label{eq:xi-random-2}\sum_{i=1}^\infty i^{-2\eta_\gamma-\eta_p}
\xi_i^{2(\alpha_w +\alpha_H +
\alpha_M + \alpha_\rho+ \beta_H\alpha_V)} &<& \infty,
\end{eqnarray}
then, Condition \ref{cond:general-noise-theorem-cond} is satisfied.
\end{proposition}
%
%}}}
%
\begin{pf} %{{{
Proposition~\ref{prop:geom-rates} implies
Condition \ref{cond:general-noise-theorem-cond}\ref{eq:poisson-bound}
with $\beta_g=\beta_H$ and $\alpha_g = \alpha_H+\alpha_M + \alpha_\rho$.
Compute
$\E[\Gamma_{i+1}]\P(\Gamma_i\neq0) = \gamma_{i+1}p_{i+1}p_i
\le ci^{-\eta_\gamma-2\eta_p}$. Then,
Proposition~\ref{prop:random-weights}
with \eqref{eq:xi-random-1} imply
Condition \ref{cond:general-noise-theorem-cond}\ref{eq:contin-sum}.

Let us then compute
$\E[\Gamma_i^2] = \gamma_i^2 p_i = c i^{-2\eta_\gamma-\eta_p}$,
and observe that
$\E[\Gamma_{i+1}\Gamma_i] =c i^{-2\eta_{\gamma}-2\eta_p}
\le c i^{-2\eta_\gamma-\eta_p}$
and that
$|\E[\Gamma_{i+1}-\Gamma_i]|
\le c i^{-\eta_{\gamma}-\eta_p-1}
\le c i^{-2\eta_{\gamma}-\eta_p}$.
With these bounds, \eqref{eq:xi-random-2} implies
Conditions \ref{cond:general-noise-theorem-cond}\ref{eq:martingale-suff}--\ref{eq:diff-consecutive-weights}.
\end{pf}
%
%}}}
%%%%%%%%%%%%%%%%%%%%
%
%re3.20 #&#
\begin{remark} %{{{
We emphasise that while our conditions on $(\Gamma_i)_{i\ge1}$
are only sufficient, it is
necessary that the random step sizes decay to zero, that is
$\limsup_{i\to\infty} \Gamma_i=0$. Otherwise, the procedure might not
converge; see \cite{roberts-rosenthal}, Example~4, for a
related result in the context of adaptive Markov chain Monte Carlo.
\end{remark}
%
%}}}

%}}}

%}}}

%%%%%%%%%%%%%%%%%%%%%%%%%%%%%%%%%%%%%%%%%%%%%%%%%%%%%%%%%%%%%%%%%%%%%%%%%%%%%%%%%
%s4 #&#
\section{Convergence}\label{sec:convergence} %{{{

Up to this point, we have only considered the stability of the
stochastic approximation process with expanding projections. Indeed,
after showing the stability we know that the projections can occur
only finitely often (almost surely), and the noise sequence can
typically be controlled. Given this, the stochastic approximation
literature provides several alternatives to show the convergence (e.g.,
\cite{benaim,benveniste-metivier-priouret,borkar-sa,chen-sa,kushner-yin-sa}).

In some special cases, one can employ our stability results directly
to establish convergence; namely, if the strict drift condition
\eqref{eq:noise-negligible} holds outside an arbitrary small
neighbourhood of the zeros of $h$.
We believe, however, that such a result has only a limited applicability,
because we suspect that it is
often useful to consider two different Lyapunov functions $w$
and $\hat{w}$ to establish the stability and convergence, respectively.

In many practical scenarios, the `true' Lyapunov function $\hat{w}$,
which would yield convergence, cannot be given in a closed form. It is
also possible that $\hat{w}$ does not satisfy Condition \ref{cond:w}
at all. We believe that it is often possible to find a simpler
`approximate Lyapunov function' $w$ satisfying Condition \ref{cond:w},
which yields a suitable drift away from the \emph{boundary} of the
space, but does not necessarily qualify as a true Lyapunov function to
establish the convergence.

We formulate below a more general convergence result
following \cite{sa-verifiable} for reader's
convenience.
%%%%%%%%%%%%%%%%%%%%
%
%co4.1 #&#
\begin{condition}\label{cond:true-mean-field} %{{{
The set $\Theta\subset\R^d$ is open, the mean field
$h\dvtx \Theta\to\R^d$ is continuous, and
there exists a continuously differentiable function
$\hat{w}\dvtx \Theta\to[0,\infty)$ such that
\begin{enumerate}[(iii)]
\item[(i)] there exists a constant $M_0>0$ such that
\[
\mathcal{L}:= \bigl\{ \theta\in\Theta\dvtx  \bigl\langle\nabla \hat{w}(\theta),h(
\theta) \bigr\rangle= 0 \bigr\} \subset\bigl\{ \theta\in\Theta\dvtx  \hat{w}(
\theta)<M_0\bigr\} ,
\]
\item[(ii)] there exists $M_1\in(M_0,\infty]$ such that
$\{\theta\in\Theta\dvtx  \hat{w}(\theta)\le M_1\}$ is compact,
\item[(iii)] for all $\theta\in\Theta\setminus\mathcal{L}$,
the inner product
$\langle\nabla\hat{w}(\theta),\hat{h}(\theta)\rangle<0$, and
\item[(iv)] the closure of $\hat{w}(\mathcal{L})$ has an empty interior.
\end{enumerate}
\end{condition}
%
%}}}
%%%%%%%%%%%%%%%%%%%%
%
%th4.2 #&#
\begin{theorem}\label{thm:convergence} %{{{
Assume Condition \ref{cond:true-mean-field} holds, and let
$\mathcal{K}\subset\Theta$ be a compact set intersecting
$\mathcal{L}$, that is,
$\mathcal{K}\cap\mathcal{L}\neq\emptyset$. Suppose that
$(\gamma_i)_{i\ge1}$ is a sequence of non-negative real numbers
satisfying $\lim_{i\to\infty} \gamma_i = 0$ and
$\sum_{i=1}^\infty\gamma_i = \infty$.
Consider the sequence
$(\theta_i)_{i\ge0}$ taking values in $\Theta$ and
defined through the recursion $\theta_i = \theta_{i-1} + \gamma_i
h(\theta_{i-1})+\gamma_i \varepsilon_i$ for all $i\ge1$,
where $(\varepsilon_i)_{i\ge1}$ take values in $\R^d$.

If there exists an integer $i_0$ such that $\{\theta_i\}_{i\ge
i_0}\subset\mathcal{K}$ and
$\lim_{m\to\infty} \sup_{n\ge m}
| \sum_{i=m}^n \gamma_i \varepsilon_i  | = 0$,
then $\lim_{n\to\infty} \inf_{x\in\mathcal{L}\cap\mathcal{K}}
|\theta_n-x| = 0$.
\end{theorem}
%
%}}}
%
\begin{pf} %{{{
Theorem~\ref{thm:convergence} is a restatement of \cite{sa-verifiable},
Theorem~2.3, but without the monotonicity assumption on the sequence
$(\gamma_i)_{i\ge1}$. The proof of
\cite{sa-verifiable}, Theorem~2.3, applies unchanged,
but the reader can also consult
\cite{andrieu-moulines-volkov}, Theorem~5, which is a slight
generalisation of Theorem~\ref{thm:convergence}.
\end{pf}
%
%}}}
%%%%%%%%%%%%%%%%%%%%
%
%re4.3 #&#
\begin{remark} %{{{
The stability results of the present paper ensure that $\theta_i$ are
eventually contained in a level set of $w$ which can usually be assumed
compact.
Then, one can take
$\mathcal{K}=\mathcal{W}_{M'}$ for some (random) $M'>0$,
and the trajectories of $(\theta_i)_{i\ge0}$ are eventually contained
within $\mathcal{K}$, and there are only finitely many projections,
almost surely. To employ Theorem~\ref{thm:convergence}, it then
suffices to show that for any $M$ in the possible range of $w$
%
%e33 #&#
\begin{equation}\label{eq:conv-noise-cond}
\lim_{m\to\infty} \sup_{n\ge m} \Biggl| \sum_{i=m}^n
\Gamma_i \bar{H}(\theta_i,X_{i+1}) \mathbb{I}
\{\theta_i\in\mathcal{W}_M\} \Biggr| = 0 .
\end{equation}
\end{remark}
%
%}}}

For the sake of completeness and
because our setting involves the random step sizes
$(\Gamma_i)_{i\ge1}$, we give a detailed theorem to establish
this noise condition, by a straightforward modification of
Theorem~\ref{thm:general-noise-theorem}.
%%%%%%%%%%%%%%%%%%%%
%
%th4.4 #&#
\begin{theorem}\label{thm:conv-noise-theorem} %{{{
Suppose that for all $i\ge1$, the step size
$\Gamma_i$ is independent of $\F_{i-1}$ and $X_i$, and the sums
$\sum_{i\ge1} \E[\Gamma_i^2]$ and $\sum_{i\ge1}
|\E[\Gamma_{i+1}-\Gamma_i]|$ are finite.
Let $\mathcal{R}\subset\hat{\Theta}$ be a compact set such
that there exists a constant $c<\infty$ so that
for any $(\theta,x)\in\mathcal{R}\times\mathsf{X}$
%
%e34 #&#
%e35 #&#
%e36 #&#
\begin{eqnarray}
\label{eq:V-exp-conv}\sup_{i\ge0} \E_{\theta,x} \bigl[V(X_{i+1})\mathbb{I}
\bigl\{A_{\mathcal
{R}}^i\bigr\} \bigr] & \le& c V(x),
\\
\label{eq:poisson-conv}\sup_{\theta\in\mathcal{R}} \bigl[ \bigl|g_\theta(x)\bigr| + \bigl|P_\theta
g_\theta(x)\bigr| \bigr] &\le& c V^{\beta_g}(x),
\\
\label{eq:poisson-diff-conv}\sum_{i=1}^\infty\E[\Gamma_{i+1}]
\E_{\theta,x} \bigl[\bigl|P_{\theta_i}g_{\theta_i}(X_i)
-P_{\theta_{i-1}}g_{\theta_{i-1}}(X_i)\bigr| \mathbb{I}\bigl
\{A_{\mathcal{R}}^i\bigr\} \bigr] &<&\infty,
\end{eqnarray}
where
$g_\theta$ is the solution of the Poisson equation as in
Proposition~\ref{prop:poisson-bound} and
$A_{\mathcal{R}}^i:=\bigcap_{n=0}^i \{\theta_n \in
\mathcal{R}\}$. Then, \eqref{eq:conv-noise-cond}
holds for $\P_{\theta,x}$-almost every $\omega\in
\bigcap_{i\ge0} A_{\mathcal{R}}^i$.
\end{theorem}
%
%}}}
The proof of Theorem~\ref{thm:conv-noise-theorem}
is given in Appendix~\ref{sec:conv-noise-theorem}.
%
%re4.5 #&#
\begin{remark} %{{{
The condition \eqref{eq:poisson-diff-conv} may be checked in practice
either with Proposition~\ref{prop:continuity} or
with Proposition~\ref{prop:nonsmooth-noise}.
To apply Theorem~\ref{thm:convergence} in the case of random step
sizes, one must check also that $\sum_{i=1}^\infty\Gamma_i$ diverges
almost surely. Assuming the conditions of Theorem~\ref{thm:conv-noise-theorem}, it is sufficient to ensure that
$\sum_{i=1}^\infty\E[\Gamma_i] = \infty$, because
$Z_n:=\sum_{i=1}^n
(\Gamma_i - \E[\Gamma_i])$ form
an a.s. convergent $L^2$-martingale.
\end{remark}
%
%}}}

%}}}

%%%%%%%%%%%%%%%%%%%%%%%%%%%%%%%%%%%%%%%%%%%%%%%%%%%%%%%%%%%%%%%%%%%%%%%%%%%%%%%%%
%s5 #&#
\section{Application:
Particle independent Metropolis--Hastings expectation maximisation}\label{sec:application} %{{{

%{{{
We consider a stochastic approximation expectation maximisation (EM)
algorithm \cite{delyon-lavielle-moulines} for static parameter
maximum likelihood
estimation in time series models, employing
a particle independent Metropolis--Hastings (PIMH) sampler
\cite{andrieu-doucet-holenstein} in order to approximate the
expectation step of the EM algorithm.
We present the generic algorithm in Section~\ref{sec:pimh-algo}. Then, we focus on a specific example
involving a Poisson count model
with an intensity determined by a latent process.
The model is given in Section~\ref{sec:pimh-model} and
the employed particle filter is discussed in Section~\ref{sec:pimh-geom}.
We establish the stability of the algorithm in Section~\ref{sec:mean-field} and conclude with a brief numerical experiment
in Section~\ref{sec:numerical}.
%}}}

%%%%%%%%%%%%%%%%%%%%%%%%%%%%%%%%%%%%%%%%
%s5.1 #&#
\subsection{Generic PIMH-EM algorithm}\label{sec:pimh-algo} %{{{

We assume a state space setting where a latent process $X_{1:n} :=
(X_1,X_2,\ldots,X_n)$ defined on some measurable space $\mathcal{X}$
gives rise to an observation process $Y_{1:n} :=
(Y_1,Y_2,\ldots,Y_n)$ taking values in a measurable space
$\mathcal{Y}$ and assumed to consist of independent random variables
given the latent process $X_{1:n}$. The process $X_{1:n}$ typically
follows a Markov model parameterised by a vector $\zeta$ taking values
in a measurable parameter space $\Xi$. The conditional marginal
distributions of the observations given the latent process are also
assumed to be parameterised by $\zeta$. This allows one to define the
so-called complete-data likelihood $p_\zeta(x_{1:n},y_{1:n})$ for any
$x_{1:n}\in\mathcal{X}^n$ and $y_{1:n}\in\mathcal{Y}^n$ and, when
applicable, the EM algorithm allows one to iteratively maximise the
likelihood $p_\zeta(y_{1:n})$. We will assume below that for any
$x_{1:n}\in\mathcal{X}^n$ and $y_{1:n}\in\mathcal{Y}^n$ there
exists a
unique parameter value $\hat{\zeta}\in\Xi$ maximising the
complete-data likelihood, which is also assumed to be uniquely
determined through a vector of sufficient statistics taking values in
an open set $\Theta\subset\R^d$.

Application of the EM algorithm requires one to compute the
expectation of the complete-data log-likelihood with respect to
$p_\zeta(\ud x_{1:n}\mid y_{1:n})$. When this is not possible
analytically one resorts to numerical methods, and we focus here on
the use of Markov chain Monte Carlo (MCMC) algorithms. More precisely, we
focus on the use of a methodology recently introduced in
\cite{andrieu-doucet-holenstein} which combines MCMC and particle
filters and is particularly well suited to sampling in state-space
models. Let us denote by $(\tilde{\mathbf{X}}, \mathbf{A})\sim
\PF(y_{1:n},\zeta)$ the full output of a particle filter
targeting the conditional distribution $p_\zeta(\ud x_{1:n}\mid
y_{1:n})$ of the model with the parameter value $\zeta$. This output
consists of all the random variables generated by the particle filter,
that is, the state variables before resampling
$\tilde{\mathbf{X}}\in\mathcal{X}^{n\times N}$ and the ancestor
indices $\mathbf{A}\in\mathbb{N}^{(n-1)\times N}$; see
\cite{andrieu-doucet-holenstein} for details. The sample trajectories
relevant to the approximation of quantities dependent on $p_\zeta(\ud
x_{1:n}\mid y_{1:n})$, denoted $X_{1:n, k}\in\mathcal{X}^{n}$
hereafter, and the associated weights $W_k\in[0,1]$ for $k=1,\ldots,N$
can be recovered from $\tilde{\mathbf{X}}$ and $\mathbf{A}$ through
functions $\bar{x}_{1:n}\dvtx \mathcal{X}^{n\times
N}\times\mathbb{N}^{(n-1)\times N }\times\mathbb{N}\to\mathcal
{X}^n$ and
$\bar{w}\dvtx \mathcal{X}^{n\times N}\times\mathbb{N}^{(n-1)\times N}
\times\mathbb{N}\to[0,1]$, such that
\begin{eqnarray*}
X_{1:n,k} := \bar{x}_{1:n}(\tilde{\mathbf{X}}, \mathbf{A}, k)
\quad \mbox{and} \quad W_k :=\bar{w}(\tilde{\mathbf{X}}, \mathbf{A}, k) .
\end{eqnarray*}
We also introduce a `sufficient statistics' function $t\dvtx \mathcal{X}^n
\times\mathcal{Y}^n \to\Theta$ which, given a set of observations
and one trajectory of the latent state variables, returns the
sufficient statistics underpinning the complete-data likelihood. From
our earlier assumption, we can define the function
$\hat{\zeta}\dvtx \Theta\to\Xi$ which returns the parameter value
maximising the conditional likelihood given some sufficient statistics
$\theta\in\Theta$.

We can now summarise our PIMH-EM algorithm with the projections
$\Pi_{\mathcal{R}_i}\dvtx \Theta\to\mathcal{R}_i$ to the sets
$\mathcal{R}_0\subset\mathcal{R}_1\subset\cdots\subset\Theta$
as follows.
%%%%%%%%%%%%%%%%%%%%
%
%al5.1 #&#
\begin{algorithm}\label{alg:pimh-em} %{{{
Choose an initial value for the
parameters $\zeta_0\in\Xi$ and set
%
%e37 #&#
%e38 #&#
\begin{eqnarray}
\label{eq:em-init-pf}\bigl(\tilde{\mathbf{X}}^{(0)}, \mathbf{A}^{(0)}_*\bigr)&
\sim& \PF(y_{1:n}, \zeta_0),
\\
\label{eq:em-init-suff}\theta_0 &:=&\Pi_{\mathcal{R}_0} \Biggl[\sum
_{k=1}^{N} W_k^{(0)} t
\bigl(X_{1:n, k}^{(0)}, y_{1:n}\bigr) \Biggr]
 .
\end{eqnarray}
For $i\ge1$, proceed recursively as follows:
%
%e39 #&#
%e40 #&#
%e41 #&#
\begin{eqnarray}
\label{eq:pimh-pf}\bigl(\tilde{\mathbf{X}}^{(i)}_*,
\mathbf{A}^{(i)}_* \bigr)& \sim& \PF \bigl(y_{1:n}, \hat{
\zeta}(\theta_{i-1}) \bigr),
\\
\label{eq:pimh-acc-rej}\bigl(\tilde{\mathbf{X}}^{(i)}, \mathbf{A}^{(i)}\bigr) & :=&
\cases{ \bigl(\tilde{\mathbf{X}}^{(i)}_*,
\mathbf{A}^{(i)}_*\bigr), & \mbox{with probability} $\min \biggl\{1,
\displaystyle \frac{\hat{Z}_{\hat{\zeta}(\theta_{i-1})}
(\tilde{\mathbf{X}}^{(i)}_*)}{
\hat{Z}_{\hat{\zeta}(\theta_{i-1})}(\tilde{\mathbf
{X}}^{(i-1)})} \biggr\}$,
\cr
\bigl(\tilde{\mathbf{X}}^{(i-1)},
\mathbf{A}^{(i-1)}\bigr), &  \mbox{otherwise,}}
\\
\label{eq:robbins-monro-pimh}\theta_{i} & :=&
\Pi_{\mathcal{R}_i} \Biggl[\theta_{i-1}+\Gamma_{i} \Biggl(
\sum_{k=1}^{N} W_k^{(i)}
t\bigl(X_{1:n, k}^{(i)},y_{1:n}\bigr) -
\theta_{i-1} \Biggr) \Biggr] ,
\end{eqnarray}
where the step \eqref{eq:pimh-acc-rej} implements an accept-reject
mechanism, and $\hat{Z}_\zeta(\hat{\mathbf{X}})$ stands for the
estimate of the likelihood $p_\zeta(y_{1:n})$ computed with the given
particles $\hat{\mathbf{X}}$ \cite{andrieu-doucet-holenstein} and
$(\Gamma_i)_{i\ge1}$ is a random step size sequence taking values in
$[0,\infty)$.
\end{algorithm}
%
%}}}

We can rewrite the steps \eqref{eq:pimh-pf} and \eqref{eq:pimh-acc-rej}
as $(\tilde{\mathbf{X}}^{(i)}, \mathbf{A}^{(i)})\sim
P_{\hat{\zeta}(\theta_{i-1})}^{\mathrm{PIMH}} ((\tilde{\mathbf
{X}}^{(i-1)},
\mathbf{A}^{(i-1)}), \cdot  )$, in terms of a Markov kernel
$P_{\zeta}^{\mathrm{PIMH}}$ with the invariant
distribution $\pi^{\mathrm{PIMH}}_{\zeta}(\ud\tilde{\mathbf
{x}},\ud
\mathbf{a})$.
As shown in \cite{andrieu-doucet-holenstein},
$\pi^{\mathrm{PIMH}}_{\zeta}(\ud\tilde{\mathbf{x}},\ud\mathbf{a})$
has the property that for any function $f\dvtx \mathcal{X}^n\rightarrow
\mathbb{R}$
\[
\int\sum_{k=1}^N \bar{w}(\tilde{
\mathbf{x}},\mathbf{a},k) f \bigl(\bar{x}_{1:n}(\tilde{\mathbf{x}},
\mathbf{a},k) \bigr) \pi^{\mathrm{PIMH}}_{\zeta}(\ud\tilde{\mathbf{x}},\ud
\mathbf{a}) = \int f(x_{1:n}) p_{\zeta}(\ud x_{1:n}|y_{1:n})
,
\]
whenever the integrals above are
well-defined. Note that it is possible to further improve on this
scheme by using smoothing
procedures within the particle filtering procedure, but we do not
consider such a possibility here.
Given this, we define $H (\theta,
(\tilde{\mathbf{x}},\mathbf{a}) ) :=
\sum_{k=1}^N \bar{w}(\tilde{\mathbf{x}},\mathbf{a},k)
t (\bar{x}_{1:n}(\tilde{\mathbf{x}},\mathbf{a},\allowbreak k) )-\theta$.
Assuming $\Pi_{\mathcal{R}_i}(\theta)=\theta$ for all
$\theta\in\mathcal{R}_i$,
we can rewrite
\eqref{eq:pimh-pf}--\eqref{eq:robbins-monro-pimh}
in our generic stochastic approximation framework as follows
%
%e42 #&#
\begin{eqnarray}\label{eq:pimh-em-reprojection}
\mathfrak{X}_{i} &\sim& P_{\theta_{i-1}}(\mathfrak{X}_{i-1}
, \cdot ),
\nonumber
\\
\theta_{i}^* &=& \theta_{i-1} + \Gamma_{i} H(
\theta_{i-1},\mathfrak{X}_{i}),
\\
\theta_{i} &=& \theta_{i}^*\mathbb{I}\bigl\{
\theta_{i}^*\in\mathcal{R}_{i}\bigr\} +
\theta_{i}^{\mathrm{proj}} \mathbb{I}\bigl\{\theta_{i}^*
\notin\mathcal {R}_{i}\bigr\}  ,\nonumber
\end{eqnarray}
where $\mathfrak{X}_i:=(\tilde{\mathbf{X}}^{(i)},\mathbf
{A}^{(i)})$ stands
for the state variable, $P_{\theta_i}:=
P_{\hat{\zeta}(\theta_{i})}^{\mathrm{PIMH}}$
and $\theta_{i}^{\mathrm{proj}}=\Pi_{\mathcal{R}_i}(\theta_i^*)$.
Note also that
the initial value $\theta_0$ computed in
\eqref{eq:em-init-pf} and \eqref{eq:em-init-suff} belongs to the
initial projection set
$\mathcal{R}_0$.
%%%%%%%%%%%%%%%%%%%%
%
%re5.2 #&#
\begin{remark} %{{{
A similar algorithm to our PIMH-EM algorithm has been independently
developed recently by Donnet and Samson \cite{donnet-samson}. They
apply the algorithm to the problem of maximum likelihood estimation of
static parameters in continuous-time diffusion models. Our work
differs in various ways: at a theoretical level, Donnet and Samson
\cite{donnet-samson} (essentially) assume a compact state space
$\mathcal{X}$, which, among other things, eliminates the need to
establish the stability of the recursion. At a methodological level,
apart from the stabilisation procedure through the expanding
projections scheme, our algorithm differs in that we use a random step
size sequence, which allows us to consider families of Markov kernels
$\{P_\theta\}_{\theta\in\Theta}$ which do not satisfy
H\"{o}lder-continuity as discussed in Section~\ref{sec:continuous-kernels}.
\end{remark}
%
%}}}

%}}}

%%%%%%%%%%%%%%%%%%%%%%%%%%%%%%%%%%%%%%%%
%s5.2 #&#
\subsection{Example: Poisson count model with random intensity}\label{sec:pimh-model} %{{{

Our specific example is a Poisson count model with an intensity determined
by a autoregressive process
\cite{zeger,chan-ledolter,fort-moulines}.
The latent stationary $\AR(1)$ process is determined by an initial
distribution $X_1\sim N(0,(1-\rho^2)^{-1}\sigma^2)$ and for $2\le
k\le
n$ through
\[
X_{k} = \rho X_{k-1} + \sigma\epsilon_{k},
\]
where $\epsilon_k$ are independent standard Gaussian random variables.
The observations are conditionally independent following the law
\[
Y_k \mid X_k \sim\Poisson \bigl(\mathrm{e}^{\alpha+ X_k} \bigr)
.
\]
For brevity, we keep $\rho\in(-1,1)$ and $\sigma^2>0$ fixed,
so that the unknown parameter of the model is
$\zeta:=\alpha\in\Xi:=\R$.

The complete data log-likelihood for the model considered satisfies
$\log (p_{\zeta}(x_{1:n},y_{1:n}) ) = L(x_{1:n},\zeta) + c$
where $c=c(\rho,\sigma^2)\in\R$ is a constant and
\[
L(x_{1:n},\zeta) := \sum_{i=1}^{n}
\bigl[y_{i}(\alpha+x_{i})-\mathrm{e}^{\alpha+x_{i}} \bigr] -
\frac{1}{2\sigma^{2}} \Biggl[ x_{1}^{2} + x_{n}^2
+ \bigl(1+\rho^{2} \bigr)\sum_{i=2}^{n-1}x_{i}^{2}
-2\rho\sum_{i=2}^{n}x_{i}x_{i-1}
\Biggr] .
\]
Let us introduce a sufficient statistics function
$t(x_{1:n},y_{1:n}) :=t(x_{1:n}) :=\sum_{i=1}^{n}\mathrm{e}^{x_i}$
taking values in $\Theta:=(0,\infty)$.
Then, denoting with
$\E_{\zeta}$ the expectation with respect to
$p_{\zeta} (\ud x_{1:n}|y_{1:n} )$,
we can write the mean field of the stochastic approximation as
\[
h(\theta)=\E_{\hat{\zeta}(\theta)} \bigl(t(X_{1:N}) \bigr)-\theta.
\]
It is straightforward to check that
the unique parameter value maximising the
complete-data likelihood is
$\hat{\zeta}(\theta):=\hat{\alpha}(\theta) =
\log (\frac{\bar{y}}{\theta} )$, where $\bar{y}:=
\sum_{i=1}^n y_i$.

%}}}

%%%%%%%%%%%%%%%%%%%%%%%%%%%%%%%%%%%%%%%%
%s5.3 #&#
\subsection{Particle filter for the example}\label{sec:pimh-geom} %{{{
\label{sec:pimh-pf}

We use the $\AR(1)$ process prior as a proposal distribution in our
particle filter, that is,
%
%e43 #&#
\begin{equation}\label{eq:proposal}
q_{\zeta}(x_i\mid x_{1:i-1}, y_{1:i})
:=p_{\zeta}(x_i\mid x_{i-1}) = N
\bigl(x_i; \rho x_{i-1}, \sigma^2\bigr) .
\end{equation}
For our convenience, we
augment the state space by adding an artificial initial state
$X_0\sim N(0,(1-\rho^2)^{-1}\sigma^2)$ with no associated observations,
which we sample perfectly.

For our analysis, we need to
quantify the dependence on $\zeta$ of the (geometric) rates of
ergodicity of
the PIMH kernel for a particular drift function.
We shall see that for this it is sufficient to upper bound the weights of
the particle filter and to lower bound the true likelihood.
%%%%%%%%%%%%%%%%%%%%
%
%pr5.3 #&#
\begin{proposition}\label{prop:weight-bound} %{{{
The weights of the particle filter for $1\le i\le n$
%
%e44 #&#
\begin{equation}\label{eq:particle-weights}
w_\zeta(x_i,x_{i-1}) :=\frac{p_\zeta(y_i\mid x_i) p_\zeta
(x_i\mid
x_{i-1})}{q_{\zeta}(x_i\mid x_{1:i-1}, y_{1:i})}
\end{equation}
with the proposal distribution
$q_{\zeta}(x_i\mid x_{1:i-1}, y_{1:i})$ given in \eqref{eq:proposal},
applied to the model described in Section~\ref{sec:pimh-model} satisfy for all $i\ge1$
%
%e45 #&#
\begin{equation}\label{eq:particle-weight-bound}
\sup_{(x_i,x_{i-1})\in\R^2} w_\zeta(x_i,x_{i-1}) \le1
.
\end{equation}
\end{proposition}
%
%}}}
%
\begin{pf} %{{{
Because we use the prior proposal, the particle weights are determined
by the likelihood. The observations are discrete, so the likelihood
is upper bounded by one.
\end{pf}
%
%}}}
%%%%%%%%%%%%%%%%%%%%
%
%pr5.4 #&#
\begin{proposition}\label{prop:likelihood-bound} %{{{
The log-likelihood of the model satisfies, with
$\bar{y}:=\sum_{i=1}^n y_i$, the bound
%
%e46 #&#
\begin{equation}\label{eq:likelihood-bound}
\log p_\zeta(y_{1:n}) \ge- \sum
_{i=1}^n \log y_i! + \bar{y} \alpha -
n\exp \biggl(\alpha+ \frac{\sigma^2}{2(1-\rho^2)} \biggr) .
\end{equation}
\end{proposition}
%
%}}}
%
\begin{pf} %{{{
We may write the log-likelihood in terms of an expectation with
respect to the stationary latent process $X_{1:n}$,
and use Jensen's inequality to obtain
%
%
%e47 #&#
%e48 #&#
\begin{eqnarray*}
\log p_\zeta(y_{1:n}) &=& \log\E \Biggl[ \prod
_{i=1}^n p(y_{i}\mid X_i,
\zeta) \Biggr] \ge\sum_{i=1}^n \E \bigl[
\log p(y_{i}\mid X_i, \zeta) \bigr]
\\
&=& \sum_{i=1}^n \E \bigl[
y_i(\alpha+ Z) - \mathrm{e}^{\alpha+ Z} - \log(y_i!) \bigr],
\end{eqnarray*}
where $Z$ follows the stationary distribution of $X_{1:n}$,
that is, $Z$ is zero-mean Gaussian with the variance
$\sigma_Z^2:=(1-\rho^2)^{-1} \sigma^2$. By recalling that the mean
of a log-Gaussian random variable $\mathrm{e}^Z$ is
$\exp (\sigma_Z^2/2)$,
we obtain the desired
bound \eqref{eq:likelihood-bound}.
\end{pf}
%
%}}}

We now turn to the particle independent Metropolis--Hastings (PIMH) kernel
in this context. Denote by $q_\zeta^{\PF}$ the overall
distribution of the random variables $(\tilde{\mathbf{X}},\mathbf{A})$
generated by the particle filter with the proposal distribution
$q_{\zeta}(x_i\mid x_{1:i-1}, y_{1:i})$ given in \eqref{eq:proposal}
and targeting $p_\zeta(x_{1:n},y_{1:n})$.
The PIMH is nothing but an ordinary
independent Metropolis--Hastings algorithm with the
proposal distribution $q_{\zeta}^{\PF}$ and the target
distribution $\pi_\zeta^{\mathrm{PIMH}}$.
%%%%%%%%%%%%%%%%%%%%

%pr5.5 #&#
\begin{proposition}\label{prop:pimh-ratio} %{{{
The ratio of the overall distribution of the particle filter and
the target density satisfies
the bound
%
%e49 #&#
\begin{equation} \label{eq:pf-bound}
\inf_{(\tilde{\mathbf{x}},\mathbf{a})\in\mathsf{X}} \frac{\ud q^{\PF}_\zeta
}{\ud\pi^{\mathrm{PIMH}}_\zeta} (\tilde{\mathbf{x}},\mathbf{a}) \ge
c_1 \exp \bigl[\bar{y}\alpha -c_2 \mathrm{e}^{\alpha}
\bigr] ,
\end{equation}
with constants $c_1=c_1(y_{1:n})>0$ and
$c_2=c_2(\rho,\sigma^2,n)>0$.
\end{proposition}
%
%}}}
%
\begin{pf} %{{{
In case of the Particle IMH, \cite{andrieu-doucet-holenstein}, page
299,
\begin{eqnarray*}
\frac{\ud\pi^{\mathrm{PIMH}}_\zeta
}{\ud q^{\PF}_\zeta} (\tilde{\mathbf{x}},\mathbf{a}) = \frac{\hat{Z}_\zeta(\tilde{\mathbf{x}},\mathbf{a})}{Z_\zeta} =
{\prod_{k=1}^n \frac{1}{N} \sum_{i=1}^N
w_\zeta \bigl( \tilde{x}_{k,i},\tilde{x}_{k-1,i}^{a}  \bigr)
}\Bigl/{p_\zeta(y_{1:n})
} ,
\end{eqnarray*}
where $N$ is the number of particles,
$w_\zeta$ are the unnormalised particle weights given in
\eqref{eq:particle-weights} and $\tilde{x}_{k,i}$
and $\tilde{x}_{k-1,i}^{a}$ stand for the $i$th particle at time $k$ and
its ancestor, respectively.
The bound \eqref{eq:pf-bound} follows directly
from the bounds \eqref{eq:particle-weight-bound} and
\eqref{eq:likelihood-bound} established in
Propositions \ref{prop:weight-bound} and
\ref{prop:likelihood-bound}, respectively.
\end{pf}
%
%}}}

The bound on the ratio of the proposal and target densities in
Proposition~\ref{prop:pimh-ratio} ensures a uniform ergodicity of the
PIMH sampler. We, however, must be able to analyse
the ergodic behaviour of the algorithm for unbounded functions.
Therefore, we consider geometric ergodicity with a certain `drift'
function $V$, which will allow us to control averages of functions $f$
such that $\sup_{x\in\mathsf{X}} |f(x)|/V(x)<\infty$.
%%%%%%%%%%%%%%%%%%%%
%
%pr5.6 #&#
\begin{proposition}\label{prop:pimh-drifts} %{{{
Let $q_\zeta^{\PF}(\ud\tilde{\mathbf{x}},\ud\mathbf{a})$
stand for the
overall proposal density of the particle filter with\vspace*{1pt} the one-step
proposal density
$q_{\zeta}(x_i\mid x_{1:i-1}, y_{1:i})$ given in \eqref{eq:proposal}
and denote
\[
V(\tilde{\mathbf{x}},\mathbf{a}) :=\sum_{i=1}^n
\sum_{j=1}^N \mathrm{e}^{2|\tilde{x}_i^{j}|} .
\]
Then, the following bounds hold
%
%e50 #&#
%e51 #&#
\begin{eqnarray}
\label{eq:proposal-V-bound}q_\zeta(V) &\le&2 n N^n \exp \biggl(\frac{2\sigma^2}{1-\rho^2}
\biggr),
\\
\label{eq:H-ratio-V}\sup_{(\tilde{\mathbf{x}},\mathbf{a}) \in\mathsf{X}} \frac{H(\theta,
(\tilde{\mathbf{x}},\mathbf{a}))}{V^{1/2}(\tilde{\mathbf
{x}},\mathbf{a})} & \le&\sqrt{n}N + \frac{|\theta|}{V^{1/2}(\tilde{\mathbf
{x}},\mathbf{a})} .
\end{eqnarray}
\end{proposition}
%
%}}}
%
\begin{pf} %{{{
The overall proposal density of the particle filter without selection
$\hat{q}_\zeta(x_{1:n})$ is in fact the finite-dimensional
distribution of the stationary $\AR(1)$ prior.
Denote by $\hat{X}_{1:n}\sim\hat{q}_\zeta$.
We obtain by a crude bound
\[
q_\zeta(V) \le\sum_{i=1}^n
N^i \E \bigl[\mathrm{e}^{2|\hat{X}_i|} \bigr] \le n N^n
\sup_{1\le i\le n} \E \bigl[\mathrm{e}^{-2\hat{X}_i} + \mathrm{e}^{2\hat
{X}_i} \bigr] .
\]
Our $\hat{X}_i$ are Gaussian with zero mean and variance
$\sigma^2/(1-\rho^2)$, and
$\E[\exp(\pm\hat{X}_i)]
=  \exp (\Var(\hat{X}_i)/2 )$.
We obtain \eqref{eq:proposal-V-bound}.

Consider then \eqref{eq:H-ratio-V}. Because $|\bar{w}|\le1$, we have
\[
\bigl|H(\theta,(\tilde{\mathbf{x}},\mathbf{a})\bigr| \le N \sup_{1\le k\le N} \bigl|t \bigl(
\bar{x}_{1:n}(\tilde{\mathbf{x}},\mathbf{a}, k) \bigr)\bigr| + |\theta| .
\]
Because $\bar{x}_{1:n}$ only chooses a path among the state variables
$\tilde{\mathbf{x}}$ and the sufficient statistics of the chosen paths
satisfy
\[
t \bigl(\bar{x}_{1:n}(\tilde{\mathbf{x}},\mathbf{a}, k)
\bigr)^2 = \Biggl(\sum_{i=1}^n
\exp \bigl( \bar{x}_{i}(\tilde{\mathbf{x}},\mathbf{a}, k) \bigr)
\Biggr)^2 \le n \sum_{i=1}^n
\exp\bigl(2\bar{x}_{i}(\tilde{\mathbf{x}},\mathbf {a}, k)\bigr) ,
\]
where $\bar{x}_{i}(\tilde{\mathbf{x}},\mathbf{a},
k)=\tilde{x}_{i,j(k,i)}$ for some integer $1\le j(k,i)\le N$.
Therefore, $|t (\bar{x}_{1:n}(\tilde{\mathbf{x}},\mathbf{a},
k) )|
\le\sqrt{n}V^{1/2}(\tilde{\mathbf{x}},\mathbf{a})$,
and we get \eqref{eq:H-ratio-V}.
\end{pf}
%
%}}}

%}}}

%%%%%%%%%%%%%%%%%%%%%%%%%%%%%%%%%%%%%%%%
%s5.4 #&#
\subsection{Stability of the PIMH-EM}\label{sec:mean-field} %{{{

We already have most of the ingredients to establish the stability of
the PIMH-EM algorithm with expanding projections applied
to our example Poisson count model with random intensity.
What remains is to identify a Lyapunov function $w$ for the sufficient
statistic.
For this purpose, we study the properties of the mean field
$h(\theta)$.
%%%%%%%%%%%%%%%%%%%%
%
%pr5.7 #&#
\begin{proposition}\label{prop:example-mean-field} %{{{
For any constant $c\in(1,\infty)$ there exists a $c_\theta=c_\theta(c,
\sigma^2,\rho, y_{1:n})\in(0,1]$ such
that
%
%e52 #&#
%e53 #&#
\begin{eqnarray}
\label{eq:h-alpha-small}h(\theta) &\ge& c \theta^{1-({1}/{2})\mathbf{1}^T \Sigma^{-1}
\mathbf{1} \log\theta} \qquad \mbox{for all } \theta\in
\bigl(0,c_\theta],
\\
\label{eq:h-alpha-big}h(\theta) &\le&-c^{-1}\theta \qquad \mbox{for all } \theta
\in[c_\theta^{-1},\infty\bigr).
\end{eqnarray}
\end{proposition}
%
%}}}
%
\begin{pf} %{{{
Observe first that
we may write, up to a constant,
\begin{eqnarray*}
p_\zeta(x_{1:n},y_{1:n}) = \det\bigl(
\Sigma^{-1/2}\bigr) \exp \Biggl( - \frac{1}{2} x_{1:n}^T
\Sigma^{-1} x_{1:n} + \sum_{i=1}^{n}
\bigl[y_{i}(\alpha+x_{i})-\mathrm{e}^{\alpha+x_{i}} \bigr] \Biggr)
,
\end{eqnarray*}
where $\Sigma^{-1}=\Sigma^{-1}(\rho,\sigma^2)\in\R^{n\times n}$
is a symmetric and positive definite matrix with
all elements equal to zero except the diagonal elements which satisfy
$\Sigma^{-1}_{1,1} =
\Sigma^{-1}_{n,n} = 1/\sigma^{2}$ and $\Sigma^{-1}_{2,2} = \cdots=
\Sigma^{-1}_{n-1,n-1} = (1+\rho^2)/\sigma^2$,
and the first diagonal above and
below the main diagonal which are such that
$\Sigma^{-1}_{i,i-1} = \Sigma^{-1}_{i-1,i} =
-\rho/\sigma^2$
for $i=2,\ldots,n$.

We may write the mean field as
%
%e54 #&#
\begin{eqnarray}\label{eq:drift-posterior-expectation}
h(\theta) &= &\int_{\R^n} \Biggl(\sum
_{i=1}^n \mathrm{e}^{x_i} - \theta \Biggr)
\frac{p_{\hat{\alpha}(\theta)}(x_{1:n},y_{1:n})}{
p_{\hat{\alpha}(\theta)}(y_{1:n})} \,\ud x_{1:n}
\nonumber
\\
&=& \theta{
\int_{\R^n} \exp \Biggl(-\frac{1}{2} x^T \Sigma^{-1} x
+ \sum_{i=1}^n y_i x_i - \bar{y}\sum_{i=1}^n
\frac{\mathrm{e}^{x_i}}{\theta} \Biggr)
\Biggl(\sum_{i=1}^n \frac{\mathrm{e}^{x_i}}{\theta} - 1 \Biggr)
\,\ud x}\\
&&{}\Bigl/{
\int_{\R^n}
\exp \Biggl(-\frac{1}{2} x^T \Sigma^{-1} x
+ \sum_{i=1}^n y_i x_i - \bar{y}\sum_{i=1}^n
\frac{\mathrm{e}^{x_i}}{\theta} \Biggr)
\,\ud x
}  .\nonumber
\end{eqnarray}
For \eqref{eq:h-alpha-big}, it is enough to observe that
by dominated convergence
$\lim_{\theta\to\infty} h(\theta)/\theta= -1$.

Let us then consider the case where
$\theta$ is small \eqref{eq:h-alpha-small}.
Denote the numerator in \eqref{eq:drift-posterior-expectation} by $N_h$,
and use the change of
variables $u_i :=e^{x_i}/\theta$ for all $i=1,\ldots,n$ to write
\begin{eqnarray*}
&&N_h = \int_{\R_+^n}  \exp \biggl(-
\frac{1}{2} (\log\theta\times\mathbf{1}+\log u)^T
\Sigma^{-1}(\log\theta\times\mathbf{1}+\log u) \biggr) \Biggl(\sum
_{i=1}^n u_i -1 \Biggr)
\\
&&\hphantom{N_h = \int_{\R_+^n}}{}\times \exp \Biggl(\sum_{i=1}^n
y_i \log(\theta u_i) - \bar{y}\sum
_{i=1}^n u_i \Biggr)
\frac{\ud u}{\prod_{i=1}^n
u_i} ,
\end{eqnarray*}
where we use the convention $\log u :=[\log u_1,\ldots,\log u_n]^T$
and $\mathbf{1}:=[1,\ldots,1]^T$.
By rearranging the terms, this can be written as
%
%e55 #&#
\begin{equation}\label{eq:nominator}
N_h = \theta^{\bar{y}
-({1}/{2})\mathbf{1}^T \Sigma^{-1}
\mathbf{1} \log\theta} \int_{\R_+^n}
\theta^{-\mathbf{1}^T \Sigma^{-1}
\log u } \Biggl( \sum_{i=1}^n
u_i -1 \Biggr) g_\Sigma(u) \,\ud u ,
\end{equation}
where the function $g_\Sigma$ is independent of $\theta$ and
for all $u\in\R_+^n$ and all $\Sigma^{-1}\in\R^{n\times n}$,
\begin{eqnarray*}
g_\Sigma(u) :=\exp \Biggl(-\frac{1}{2} \log u^T
\Sigma^{-1} \log u + \sum_{i=1}^n
(y_i-1) \log u_i - \bar{y}\sum
_{i=1}^n u_i \Biggr) >0 .
\end{eqnarray*}
We shall partition the domain $\R_+^n$ according to the sign of the
integrand in \eqref{eq:nominator} as $I_-
:=\{u\in\R_+^n\dvt \sum_{i=1}^n u_i < 1\}$ and $I_+ :=
\R_+^n\setminus I_-$.
Observe that for all $u\in I_-$, the elements of $\log u$ are
all negative, and the row sums of $\Sigma^{-1}$ are
all positive. Therefore, $-\mathbf{1}^T \Sigma^{-1} \log u > 0$ for
all $u\in I_-$ and because the integral is finite for any
fixed $\theta>0$,
\[
\lim_{\theta\to0+} \int_{I_-} \theta^{-\mathbf{1}^T \Sigma^{-1}
\log u }
\Biggl( \sum_{i=1}^n u_i -1
\Biggr) g_\Sigma(u)\, \ud u = 0 .
\]
On the other hand, considering the subset $\hat{I}_+ :=\{u\in\R_+^n\dvt
\forall i=1,\ldots,n\ \log(u_i)>0 \}\subset I_+$, then similarly
$-\mathbf{1}^T \Sigma^{-1} \log u < 0$ for all $u\in\hat{I}_+$,
whence
\[
\lim_{\theta\to0+} \int_{\hat{I}_+} \theta^{-\mathbf{1}^T \Sigma^{-1}
\log u }
\Biggl( \sum_{i=1}^n u_i -1
\Biggr) g_\Sigma(u)\, \ud u = \infty.
\]
Overall, we deduce that for any constant $c'>0$ there exists a
$c_\theta=c_\theta(c',\Sigma,y_{1:n})>0$ such that for all
$\theta\in(0,c_\theta)$,
\[
N_h \ge c' c_\Sigma\theta^{
\bar{y} -({1}/{2})\mathbf{1}^T \Sigma^{-1}
\mathbf{1} \log\theta} >
0 .
\]

We are left with upper bounding the denominator $D_h$ in
\eqref{eq:drift-posterior-expectation}, which we write as an
expectation with respect to a random variable $X\sim N(0,\Sigma)$
\[
D_h = c_\Sigma\E \Biggl[ \exp \Biggl(\sum
_{i=1}^n y_i X_i -
\frac{\bar{y}}{\theta} \sum_{i=1}^n
\mathrm{e}^{X_i} \Biggr) \Biggr] .
\]
By elementary calculus, one can compute that for $y,\bar{y},\theta>0$
\[
\sup_{x\in\R} \exp \biggl(yx-\frac{\bar{y}}{\theta}\mathrm{e}^x \biggr) =
\theta^y \exp \biggl( y \log\frac{y}{\bar{y}} - y \biggr) ,
\]
so $D_h \le c_{y_{1:n},\Sigma} \theta^{\bar{y}}$, and
we deduce \eqref{eq:h-alpha-small}
by choosing $c'$ sufficiently large.
\end{pf}
%
%}}}

Now we are ready to establish the stability of the PIMH-EM in our
example setting.
%%%%%%%%%%%%%%%%%%%%
%
%pr5.8 #&#
\begin{proposition}\label{prop:pimh-em-stability} %{{{
Consider Algorithm \ref{alg:pimh-em} applied to the model specified in
Section~\ref{sec:pimh-model}, with the projections
\eqref{eq:pimh-em-reprojection}. The projection sets
are defined as $\mathcal{R}_i:=
\{\theta\in\Theta\dvt  \underline{\theta}_i\le\theta\le\bar{\theta
}_i\}$
and the projections as
$\theta_i^\mathrm{proj}:=(\underline{\theta}_i\vee\theta_i^*)
\wedge\bar{\theta}_i$, with the constant sequences
$\underline{\theta}_i\downarrow0$ and $\bar{\theta}_i\uparrow
\infty$
satisfying
\[
\liminf_{i\to\infty} \underline{\theta}_i \log(i) =\infty
\quad \mbox{and}\quad  \limsup_{i\to\infty} \frac{\bar{\theta}_i}{i^{\epsilon}} = 0
\]
for all $\epsilon>0$.
The step sizes are defined as $\Gamma_i :=c_\gamma i^{-\eta_\gamma} \mathbb{I}\{U_i \le c_p
i^{-\eta_p}\}$
where $c_\gamma,c_p\in(0,\infty)$, and the constants
$\eta_\gamma,\eta_p\in(0,1)$ satisfy
$\eta_\gamma+\eta_p<1$, $2\eta_\gamma+\eta_p>1$ and
$\eta_\gamma+2\eta_p>1$, and
$(U_i)_{i\ge1}$ are uniform $(0,1)$ distributed
random variables independent on the history $\F_{i-1}$ and $X_i$.

Then, there exists $0<c_1< c_2<\infty$ such that
for any $(\theta,x)\in\mathcal{R}_0\times\mathsf{X}$,
\[
\P_{\theta,x} \Biggl(\bigcup_{m=1}^\infty
\bigcap_{n=m}^\infty \{c_1 \le
\theta_i \le c_2 \} \Biggr) = 1 .
\]
\end{proposition}
%
%}}}
%
\begin{pf} %{{{
Let $c_\theta\in(0,1)$ be the constant from Proposition~\ref{prop:example-mean-field} applied with, say, $c=1$, and
define $\hat{w}(\theta) :=
|\theta-c_\theta^*|$ with $c_\theta^*:=(c_\theta+c_\theta^{-1})/2$.
Define $w$ as the smoothed version of $\hat{w}$ through
the convolution $w:=\hat{w}*\phi$
with a $C^\infty$-mollifier $\phi$ supported on a sufficiently small
$[-\epsilon_\phi,\epsilon_\phi]$, so that
$w=\hat{w}$ on $(0,c_\theta]\cup[c_\theta^{-1},\infty)$.
Then, $w$ is twice differentiable with bounded derivatives,
$w(\theta)< w(\theta')$ for all $\theta\in
\mathcal{W}_{M_0} = [c_\theta,c_\theta^{-1}]$ and
$\theta'\in\R\setminus\mathcal{W}_{M_0}$, where
$M_0:=c_\theta^*-c_\theta>0$.
To sum up, letting $\xi_i:=i\vee1$ for $i\ge0$,
Conditions \ref{cond:w}\ref{item:w-hessian},
\ref{item:projection-sets},
\ref{item:proj} and \ref{item:xi} hold with
$\alpha_w=0$ and with some constant $c<\infty$.

Now, we turn into establishing Condition \ref{cond:stringent-drift}.
The bounds from Proposition~\ref{prop:example-mean-field}
imply
$\delta:=\inf_{\theta\ge c_\theta} - \langle
h(\theta),\nabla w (\theta)\rangle> 0$
and
\begin{eqnarray*}
\delta_i &:=& \inf_{\theta\in[\underline{\theta}_i, c_\theta^{-1}]} - \bigl\langle h(\theta),\nabla w
(\theta)\bigr\rangle \ge c\inf_{\theta\in[\underline{\theta}_i, c_\theta^{-1}]} \theta^{1-c_h \log(\theta)}
\\
&=& c \underline{\theta}_i^{1-c_h \log(\underline{\theta}_i)} \ge c_1 (\log
i)^{-c_2 \log\log i}
\end{eqnarray*}
for $i\ge2$, where $c_1,c_2\in(0,\infty)$.
Therefore, with our choice of the step
sizes
$\sum_{i=1}^\infty(\delta\wedge\delta_i) \E[\Gamma_i] = \infty$,
implying that $\sum_{i=1}^\infty(\delta\wedge\delta_i) \Gamma_i =
\infty$
almost surely.\footnote{The random variables $Z_n:=\sum_{i=1}^n
(\delta\wedge\delta_i)
(\Gamma_i - \E[\Gamma_i])$ form an a.s. convergent $L^2$-martingale.}

Recalling that $\hat{\alpha}(\theta) = \log(\bar{y}/\theta)$,
we bound by Proposition~\ref{prop:pimh-ratio}
\[
\hat{\epsilon}(\theta) := \inf_{(\tilde{\mathbf{x}},\mathbf{a})\in\mathsf{X}} \frac{\ud q^{\PF}_{\hat{\alpha}(\theta)}
}{\ud\pi^{\mathrm{PIMH}}_{\hat{\alpha}(\theta)} } (\tilde{
\mathbf{x}},\mathbf{a}) \ge c_1 \biggl( \frac{\mathrm{e}^{-c_2/\theta}}{\theta}
\biggr)^{\bar{y}} ,
\]
where $c_1,c_2<\infty$ are constants independent of $\theta$.
Now, fix an $\varepsilon>0$. Then, it is straightforward to check that
there exists a constant $c<\infty$ such that
for all $i\ge1$
\[
\sup_{\theta\in\mathcal{R}_i} \frac{1}{\hat{\epsilon}(\theta)} = \biggl(\sup_{\theta\in[\underline{\theta}_i, 1]}
\frac{1}{\hat{\epsilon}(\theta)} \biggr)\vee \biggl(\sup_{\theta\in[1,\bar{\theta}_i]} \frac{1}{\hat{\epsilon}(\theta)}
\biggr) \le c \xi_i^{\varepsilon} .
\]
Without loss of generality, we may assume
$\hat{\epsilon}(\theta)\le1/2$, so
Corollary~\ref{cor:imh-bound-geometric} implies that the $P_{\theta}$
is geometrically ergodic with constants
$\hat{M}=\hat{M}(\hat{\epsilon}(\theta))=c\hat{\epsilon
}^{-2}(\theta)$
and
$\hat{\rho}=\hat{\rho}(\hat{\epsilon}(\theta))=(1-\hat{\epsilon
}(\theta)/2)$.
It is easy to see that then Condition \ref{cond:geom-rates} holds with
$\alpha_M = 2\varepsilon$ and
$\alpha_\rho=\varepsilon$.

Let $V$ be defined as in
Proposition~\ref{prop:pimh-drifts}. Then, there exists a
constant $c<\infty$ such that
\[
\sup_{\theta\in\mathcal{R}_i} \bigl\| H(\theta, \cdot )\bigr\|_{V^{1/2}} \le c_2
+ \sup_{\theta\in\mathcal{R}_i} |\theta| = c_2 + \bar{\theta}_i
\le c \xi_i^{\varepsilon} ,
\]
implying Condition
\ref{cond:general-noise-theorem-cond}\ref{eq:H-bound} with $\beta_H
= 1/2$ and $\alpha_H = \varepsilon$.
The drift condition assumed in Lemma~\ref{lemma:drift} holds with
$\lambda_i=1-\inf_{\theta\in\mathcal{R}_i}\hat{\epsilon}(\theta
)$ and
$b_i=b<\infty$ due to
Corollary~\ref{cor:imh-bound-geometric}. This implies
Condition \ref{cond:general-noise-theorem-cond}\ref{eq:V-exp-bound} with $\alpha_V = \alpha_\rho= \varepsilon$.

Now, Proposition~\ref{prop:nonsmooth-noise} is applicable as soon as
we choose $\varepsilon>0$ above sufficiently small so that
\[
\alpha_w + \alpha_M + \alpha_\rho+
\alpha_H + \beta_H \alpha_V < (
\eta_\gamma+2\eta_p-1)\wedge\frac{2\eta_\gamma+\eta_p-1}{2} .
\]
Proposition~\ref{prop:nonsmooth-noise} implies
Condition \ref{cond:general-noise-theorem-cond}, allowing us to
establish the noise condition in Theorem~\ref{thm:general-noise-theorem}.
Finally, Theorem~\ref{th:bounded-w-stability} yields the claim with
$c_1 = c_\theta$ and $c_2 = c_{\theta}^{-1}$.
\end{pf}
%
%}}}
We remark that the condition for $\bar{\theta}_i$ in Proposition~\ref{prop:pimh-em-stability} can be relaxed by only assuming it to
hold with a certain fixed $\epsilon>0$ depending on $\bar{y}$,
$\eta_\gamma$ and $\eta_p$.

%}}}

%
%f2 #&#
\begin{figure}[b]

\includegraphics{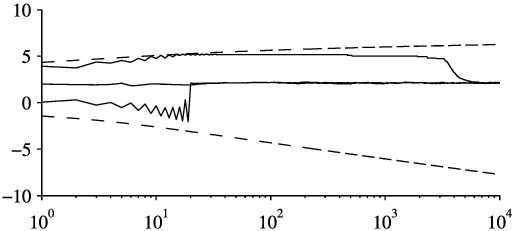}

\caption{Trajectories of the estimate $\hat{\alpha}(\theta_i)$
corresponding the PIMH-EM started from three different
initial values for $\hat{\alpha}_0$. The dashed lines correspond to
the boundaries induced to $\hat{\alpha}(\theta_i)$ by
$(\underline{\theta}_i)_{i\ge0}$ and $(\bar{\theta}_i)_{i\ge0}$.
Notice the logarithmic scale on the $x$-axis (iterations).}
\label{fig:trajectories}
\end{figure}
%

%%%%%%%%%%%%%%%%%%%%%%%%%%%%%%%%%%%%%%%%
%s5.5 #&#
\subsection{Numerical experiment}\label{sec:numerical} %{{{

We illustrate our algorithm briefly in practice in the setup of Proposition~\ref{prop:pimh-em-stability}. We consider the same setting
as Fort and Moulines \cite{fort-moulines}: we have
$n=100$ simulated observations of
the model of Section~\ref{sec:pimh-model}
with parameters $\alpha=2$, $\rho= 0.4$ and $\sigma^2=1$.

We use the following projection sequences to control the sufficient
statistic
\[
\underline{\theta}_i :=\underline{c}\log^{\underline{\epsilon}-1}(i+2)
\quad \mbox{and} \quad \bar{\theta}_i :=\bar{c}_1
(i+2)^{\bar{c}_2/\log^{\bar{\epsilon}}(i+2)} ,
\]
with the constants
$\underline{c}=0.1 m_\theta$, $\bar{c}_1 = 10 m_\theta$,
$\underline{\epsilon}=\bar{\epsilon} = 0.1$ and $\bar{c}_2 = 1$,
where
$m_\theta:=n\exp (\frac{\sigma^2}{2(1-\rho^2)} )$
is the prior expectation
of the sufficient statistic.
The step size sequence parameters are $c_\gamma=6$, $c_p = 3$ and
$\gamma_\eta=\gamma_p = 0.35$. The number of particles is set to
$N=1000$.

Figure~\ref{fig:trajectories} shows the trajectories of
the estimates $\hat{\alpha}(\theta_i)$ for
$10\mathord{,}000$ iterations of the algorithm starting from three
different initial values $\hat{\alpha}_0 \in\{0,2,4\}$. The final values
of the estimates $\hat{\alpha}$ are within 2.10--2.16. The average
acceptance rate during the runs varied between 46--72\%.
Notice the unstable initial behaviour of
the estimates in Figure~\ref{fig:trajectories},
which is controlled by the projections.
\begin{appendix}
%%%%%%%%%%%%%%%%%%%%%%%%%%%%%%%%%%%%%%%%%%%%%%%%%%%%%%%%%%%%%%%%%%%%%%%%%%%%%%%%
%s6 #&#
\section{Geometric ergodicity from drift condition}\label{sec:drift}
Before the proof of Proposition~\ref{prop:drift-to-erg},
we restate the result by Meyn and
Tweedie \cite{meyn-tweedie-computable}
upon which the proof relies.
%%%%%%%%%%%%%%%%%%%%
%
%th6.1 #&#
\begin{theorem}[(Meyn and Tweedie
\cite{meyn-tweedie-computable} Theorem~2.3)]\label{th:m-t-computable} %{{{
Suppose Condition \ref{cond:drift-mino} holds.
Then, for all $k\ge0$ and $\|f\|_{V}<\infty$
\begin{eqnarray*}
\bigl|P_s^k(x,f)-\pi(f)\bigr| \le V(x) (1+\gamma)
\frac{\rho}{\rho-\vartheta}\rho^k \|f\|_{V}
\end{eqnarray*}
for any $\rho>\vartheta=1-\tilde{M}^{-1}$, for
\begin{eqnarray*}
\tilde{M} = \frac{1}{(1-\check{\lambda})^2} \bigl[1-\check{\lambda}+\check {b}+
\check{b}^2 + \bar{\zeta} \bigl( \check{b}(1-\check{\lambda})+ \check
{b}^2 \bigr) \bigr] ,
\end{eqnarray*}
defined in terms of
\[
\gamma= \delta^{-2} [4b+2\delta\lambda v ],\qquad  \check{\lambda} = (
\lambda+\gamma)/(1+\gamma)<1 \quad \mbox{and}\quad  \check{b} = v + \gamma< \infty,
\]
and the bound
\begin{eqnarray*}
\bar{\zeta} \le \frac{4-\delta^2}{\delta^5} \biggl(\frac{b}{1-\lambda}
\biggr)^2 .
\end{eqnarray*}
\end{theorem}
%
%}}}
\begin{pf*}{Proof of Proposition~\ref{prop:drift-to-erg}} %{{{
Let us first consider the claim for $r=1$.
Define first
\[
\bar{\zeta} :=\bigl(4-\delta^2\bigr) \delta^{-5}
b^2 (1-\lambda)^{-2} \le4 \delta^{-5}
\bar{b}^2 (1-\lambda)^{-2} ,
\]
and observe that $\gamma:=\delta^{-2}[4b + 2\delta\lambda v]
\le6 \delta^{-2} \bar{b}$. We also have
\[
\check{\lambda}:= \frac{\lambda+\gamma}{1+\gamma} \le\frac{\lambda+ 6 \delta^{-2}
\bar{b}}{1+6 \delta^{-2} \bar{b}} \quad \mbox{implying}\quad
\frac{1}{1-\check{\lambda}} \le\frac{1+6 \delta^{-2}
\bar{b}}{1-\lambda} \le\frac{7\delta^{-2}\bar{b}}{1-\lambda} .
\]
We have also $\check{b} :=v + \gamma\le7 \delta^{-2} \bar{b}$.
Now, we can bound
\begin{eqnarray*}
\tilde{M}& :=&\frac{1}{(1-\check{\lambda})^2} \bigl[1-\check{\lambda} + \check{b} +
\check{b}^2 + \bar{\zeta} \bigl(\check{b}(1-\check{\lambda}) +
\check{b}^2 \bigr) \bigr]
\\
&\le&\frac{1}{(1-\check{\lambda})^2} \bar{\zeta}\bigl(5\check{b}^2\bigr) \le48\mbox{,}020
(1-\lambda)^{-4}\delta^{-13}\bar{b}^6 .
\end{eqnarray*}
Now we can take $\rho_1 :=1 -  [100\mbox{,}000
(1-\lambda)^{-4}\delta^{-13}\bar{b}^6 ]^{-1}$
satisfying $\rho_1 > 1-\tilde{M}^{-1}/2$. Finally, the claim holds
with $c_1^*=c^*:=336\mbox{,}140$ by setting
\[
M_1 :=(1+\gamma)\frac{\rho}{\rho-(1-\tilde{M}^{-1})} \le(1+\gamma) 2\tilde{M}
\le336\mbox{,}140 (1-\lambda)^{-4} \delta^{-15}\bar{b}^7 .
\]

Let us consider then the case $r\in(0,1)$.
Observe first that by Jensen's inequality
\begin{eqnarray*}
P V^r(x) &\le& \bigl(P V(x) \bigr)^r \le
\lambda^r V^r(x) \qquad \mbox{for all } x\notin C,
\\
P V^r(x) &\le& \Bigl(\sup_{z\in C} V(z) + b
\Bigr)^r \le2^r (v\vee b)^r \qquad \mbox{for all }
x\in C.
\end{eqnarray*}
That is, Condition \ref{cond:drift-mino} holds for $V^r$
with $\lambda_r :=\lambda^r$,
$\bar{b}_r:=2 \bar{b}^r$, and
$v_r:=\sup_{x\in C} V^r(x) =  (\sup_{x\in C}
V(x) )^r = v^r$. Because $t\mapsto t^r$ is concave,
$\lambda^r \le1 - r(1-\lambda)$ and so $(1-\lambda^r)^{-1} \le r^{-1}
(1-\lambda)^{-1}$. We may take $c_r^* :=(2r^{-1})^4 c^*$.
\end{pf*}
%
%}}}

%}}}
%%%%%%%%%%%%%%%%%%%%%%%%%%%%%%%%%%%%%%%%%%%%%%%%%%%%%%%%%%%%%%%%%%%%%%%%%%%%%%%%%
%s7 #&#
\section{Noise condition for convergence theorem}\label{sec:conv-noise-theorem}\vspace*{-10pt}
%%%%%%%%%%%%%%%%%%
%
\begin{pf*}{Proof of Theorem~\ref{thm:conv-noise-theorem}} %{{{
We give only the required modifications to the proof of
Theorem~\ref{thm:general-noise-theorem} regarding
\eqref{eq:noise-negligible-poisson}. First, by symbolically
substituting
$\nabla{w}\equiv1$, it is sufficient to show that
claim holds for the following four terms in turn:
\begin{eqnarray*}
R_{m,n}^{1} & :=&\sum_{i=m}^{n}
\Gamma_{i+1} \bigl( g_{\theta_{i}}(X_{i+1})-P_{\theta_{i}}g_{\theta_{i}}(X_{i})
\bigr) \mathbb{I}\bigl\{A_{\mathcal{R}}^i\bigr\},
\\
R_{m,n}^{2} & :=&\sum_{i=m}^{n}
\Gamma_{i+1} \bigl( P_{\theta_{i}}g_{\theta_{i}}(X_{i})
-P_{\theta_{i-1}}g_{\theta_{i-1}}(X_{i}) \bigr) \mathbb{I}\bigl
\{A_\mathcal{R}^i\bigr\},
\\
R_{m,n}^{4} & := &\bigl(\Gamma_{m}
P_{\theta_{m-1}}g_{\theta_{m-1}}(X_{m}) - \Gamma_{n+1}
P_{\theta_{n}}g_{\theta_{n}}(X_{n+1}) \bigr) \mathbb{I}\bigl
\{A_\mathcal{R}^n\bigr\},
\\
R_{m,n}^{5} & :=&\sum_{i=m}^{n}(
\Gamma_{i+1}-\Gamma_{i}) P_{\theta_{i-1}}g_{\theta_{i-1}}(X_{i})
\mathbb{I}\bigl\{A_\mathcal{R}^{i-1}\bigr\} .
\end{eqnarray*}
The first term $R_{m,n}^{1}$ is a martingale, so
by Doob's inequality, \eqref{eq:V-exp-conv} and \eqref{eq:poisson-conv},
\begin{eqnarray*}
\Bigl(\E_{\theta,x} \Bigl[\sup_{n\ge m}\bigl|R_{m,n}^{1}\bigr|
\Bigr] \Bigr)^2 &\le& C \sum_{i=m}^\infty
\E_{\theta,x}\bigl[\Gamma_{i+1}^2 \bigl|
g_{\theta_{i}}(X_{i+1})-P_{\theta_{i}}g_{\theta_{i}}(X_{i})\bigr|^2
\mathbb{I}\bigl\{A_{\mathcal{R}}^i\bigr\} \bigr]
\\
&\le& CV^{2\beta_g}(x)\sum_{i=m}^{\infty}\E
\bigl[\Gamma_{i+1}^2\bigr] \mathop{\xrightarrow}^{m\to\infty}
0 .
\end{eqnarray*}
The claim for the second term is implied directly by
\eqref{eq:poisson-diff-conv}. For the term $R_{m,n}^{4}$, it is enough
to observe that
\[
\E_{\theta,x} \Bigl[\sup_{n\ge m} \bigl(R_{m,n}^{4}
\bigr)^2 \Bigr] \le 4 \sum_{i=m}^\infty
\E\bigl[\Gamma_i^2\bigr] \E_{\theta,x} \bigl[\bigl|
P_{\theta_{i-1}}g_{\theta_{i-1}}(X_{i})\bigr|^2 \mathbb{I}
\bigl\{A_\mathcal{R}^{i-1}\bigr\} \bigr] \le C
V^{2\beta_g}(x) \sum_{i=m}^\infty\E
\bigl[\Gamma_{i}^2\bigr] .
\]
Finally, we may employ Lemma~\ref{lemma:random-weights}
for $R_{m,n}^5$ with $U_i:=\Gamma_i$ and
$B_{i-1}\break  :=|P_{\theta_{i-1}}g_{\theta_{i-1}}(X_i)|
\mathbb{I}\{A_\mathcal{R}^{i-1}\}$ because
$\E_{\theta,x}[|B_{i-1}|]
\le C V^{\beta_g}(x) $ and
$\E_{\theta,x}[B_{i-1}^2]
\le C V^{2\beta_g}(x) $.
\end{pf*}
%
%}}}

%}}}

%%%%%%%%%%%%%%%%%%%%%%%%%%%%%%
%s8 #&#
\section{Geometric ergodicity of IMH}\label{sec:imh} %{{{

We provide here quantitative bounds for the ergodicity constants
for independent Metropolis--Hastings kernels. To our knowledge, the
results here are new, and can be useful also in other settings.

Recall that the independent Metropolis--Hastings kernel with target
density $\pi$ and proposal density $q$ on space $\mathsf{X}\subset\R^d$
is defined as
\[
P(x,A) :=\int_A \alpha(x,y) q(y) \,\ud y + \mathbb{I}\{x
\in A\} \biggl(1-\int_{\mathsf{X}} \alpha(x,y) q(y) \,\ud y \biggr)
\]
for all $x\in\mathsf{X}$ and measurable $A\subset\mathsf{X}$,
where the acceptance probability
$\alpha(x,y)$ is defined as
\[
\alpha(x,y) :=\min \biggl\{1, \frac{\pi(y)/q(y)}{\pi(x)/q(x)} \biggr\} .
\]
%%%%%%%%%%%%%%%%%%%%

%pr8.1 #&#
\begin{proposition}\label{prop:imh-bound} %{{{
Assume $P$ is the independent Metropolis--Hastings kernel with target
density $\pi$ and proposal density $q$ satisfying
$\epsilon:=\inf_{x\in\mathsf{X}}
q(x)/\pi(x)>0$. Let $V\dvtx \mathsf{X}\to[1,\infty)$ be a function
with $q(V)<\infty$. Then,
\renewcommand\theenumi{(\roman{enumi})}
\renewcommand\labelenumi{\theenumi}
\begin{enumerate}[(ii)]
\item\label{item:imh-drift}
the drift inequality
\begin{eqnarray*}
P V(x) \le\rho V(x) + q(V) \qquad \mbox{for all } x\in\mathsf{X}
\end{eqnarray*}
holds with the constant $\rho:=1 - \epsilon$, and
\item\label{item:imh-geom-rate}
the following bound holds for
any measurable function $f\dvtx \mathsf{X}\to\R^d$ with
$\|f\|_V:=\sup_{x\in\mathsf{X}} |f(x)|/V(x) < \infty$,
all
$k\ge1$ and all $x\in\mathsf{X}$
\begin{eqnarray*}
\bigl| P^k f(x) - \pi(f) \bigr| \le k M (1-\epsilon)^k \|f
\|_V V(x),
\end{eqnarray*}
where the constant $M=q(V)[1 + \epsilon^{-1} + (1-\epsilon)^{-1}]$.
\end{enumerate}
\end{proposition}
%
%}}}
%
\begin{pf} %{{{
Denote by $r(x) :=\pi(x)/q(x)$ so that $\alpha(x,y) =
\min\{1,r(x)/r(y)\}$ and compute
\[
\frac{P V(x)}{V(x)} -1 = \frac{\int V(y) \alpha(x,y) q(y) \,\ud y}{V(x)} - \int\min\bigl\{
r^{-1}(y), r^{-1}(x) \bigr\} \pi(y) \,\ud y \le
\frac{q(V)}{V(x)} - \epsilon.
\]
This readily implies \ref{item:imh-drift}.\vadjust{\goodbreak}

Observe then that for any measurable $A\subset\mathsf{X}$, the
following uniform minorisation inequality holds
\[
P(x,A) \ge\int_A \alpha(x,y) q(y)\, \ud y \ge\epsilon
\pi(A) .
\]
By this inequality, one can define a Markov kernel $Q(x,A) :=
(1-\epsilon)^{-1} (P(x,A)-\epsilon\pi(A) )$. By
\ref{item:imh-drift}, we have $Q V(x) \le(1-\epsilon)^{-1}
(\rho
V(x) + q(V) ) = V(x) + (1-\epsilon)^{-1}q(V)$ so by induction we obtain
\[
Q^k V(x) \le V(x) + k(1-\epsilon)^{-1} q(V) .
\]
Observe that for any probability measure $\nu$ with $\nu(V)<\infty$,
one has $\nu(|f|) \le\|f\|_V \nu(V)$, and that
\[
\pi(V) = \int\frac{\pi(x)}{q(x)} V(x) q(x) \,\ud x \le\epsilon^{-1} q(V)
.
\]
Note that $\pi Q = \pi$, whence by denoting $\Pi(x, \cdot ) :=\pi
( \cdot )$
one can compute for any $k\ge1$
\begin{eqnarray*}
\bigl|P^{k} f(x) - \pi(f)\bigr| &=& \bigl|(P-\Pi) P^{k-1}f(x)\bigr| = (1-
\epsilon)\bigl|(Q-\Pi) P^{k-1}f(x)\bigr|
\\
&=& (1-\epsilon)\bigl| QP^{k-1} f(x) - \pi(f)\bigr| = {\cdots} = (1-
\epsilon)^k\bigl|Q^k f(x) - \pi(f)\bigr|
\\
&\le&(1-\epsilon)^k \bigl(V(x) + k(1-\epsilon)^{-1} +
\epsilon^{-1} \bigr)\|f\|_V q(V) , %
\end{eqnarray*}
establishing \ref{item:imh-geom-rate}.
\end{pf}
%
%}}}
%%%%%%%%%%%%%%%%%%%%
%
%co8.2 #&#
\begin{corollary}\label{cor:imh-bound-geometric} %{{{
In Proposition~\textup{\ref{prop:imh-bound}}, the bound
\textup{\ref{item:imh-geom-rate}}
can be replaced with the following
\[
\bigl| P^k f(x) - \pi(f) \bigr| \le M' (1-\zeta
\epsilon)^k \|f\|_V V(x) ,
\]
where $\zeta\in(0,1)$ can be chosen arbitrarily and where
\[
M' = \frac{M}{\mathrm{e}} \biggl[ \log \biggl(\frac{1-\zeta\epsilon}{1-\epsilon}
\biggr) \biggr]^{-1} .
\]
If $\epsilon\le1/2$, then $M'$ can be taken as $M' =
2M[\mathrm{e}(1-\zeta)\epsilon]^{-1}$.
\end{corollary}
%
%}}}
%
\begin{pf} %{{{
From Proposition~\ref{prop:imh-bound}, we obtain
%
%
%e56 #&#
%e57 #&#
\begin{eqnarray*}
\bigl| P^k f(x) - \pi(f) \bigr| &\le& k M (1-\epsilon)^k \|f
\|_V V(x)
\\
&\le& M' (1-\zeta\epsilon)^k \|f\|_V
V(x) ,
\end{eqnarray*}
with
\[
M' :=M \sup_{k\ge1} k \biggl(\frac{1-\epsilon}{1-\zeta\epsilon}
\biggr)^k \le\frac{M}{\mathrm{e}} \biggl[ \log \biggl(\frac{1-\zeta\epsilon}{1-\epsilon}
\biggr) \biggr]^{-1} ,
\]
since by a straightforward calculation one obtains for any $a\in(0,1)$
that $\sup_{x>0} x a^x =  (\mathrm{e} \log(1/a) )^{-1}$.
Suppose then that $\epsilon\le1/2$ and
notice that for any $h> 0$ one has
$\log(1+h) \ge h - \frac{1}{2}h^2$ and so
\[
\log \biggl(\frac{1-\zeta\epsilon}{1-\epsilon} \biggr) \ge\frac{(1-\zeta)\epsilon}{1-\epsilon} \biggl(1-
\frac{1}{2} \frac{(1-\zeta)\epsilon}{1-\epsilon} \biggr) \ge\frac{1}{2}(1-\zeta)
\epsilon.
\]\upqed
\end{pf}
%
%}}}

%}}}

%%%%%%%%%%%%%%%%%%%%%%%%%%%%%%
%s9 #&#
\section{Nomenclature}\label{sec:nom} %{{{
\begin{itemize}
\item$\alpha_w$ in Condition \ref{cond:w}, page \pageref{cond:w},
related to
the growth of $\sup_{\theta\in\mathcal{R}_i}|\nabla w(\theta)|$.
\item$\alpha_H,\beta_H$ in Condition \ref{cond:general-noise-theorem-cond},
page \pageref{cond:general-noise-theorem-cond}, characterise
$\sup_{\theta\in\mathcal{R}_i} |H(\theta,x)| $.
\item$\alpha_V,\beta_V$ in Condition \ref{cond:general-noise-theorem-cond},
page \pageref{cond:general-noise-theorem-cond}, characterise
$\E_{\theta,x}[V(X_i)] $.
\item$\alpha_g,\beta_g$ in Condition \ref{cond:general-noise-theorem-cond},
page \pageref{cond:general-noise-theorem-cond}, characterise
$\sup_{\theta\in\mathcal{R}_i} [
|g_\theta(x)| + |P_\theta g_\theta(x)|] $.
\item$\beta_D$ in Condition \ref{cond:continuity},
page \pageref{cond:continuity}, characterises
the H\"older continuity of
$\llVert  P_{\theta}f-P_{\theta^{\prime}}f\rrVert_{V^r}$.
\item$\alpha_\Delta,\beta_\Delta$ in Condition \ref{cond:contin-poisson},
page \pageref{cond:contin-poisson}, characterise the size of
$\sup_{(\theta,\theta')\in\mathcal{R}_i^2} \llVert  H(\theta
, \cdot )
-H(\theta^{\prime}, \cdot )\rrVert_{V^{\beta_H}}$.
\item$\alpha_M$ and $\alpha_\rho$ are defined in Condition
\ref{cond:geom-rates}, page \pageref{cond:geom-rates}, and
characterise the loss of ergodicity through the growth of
geometric ergodicity constants
$\sup_{\theta\in\mathcal{R}_i} M_{\theta,r} $ and
$\sup_{\theta\in\mathcal{R}_i} (1-\rho_{\theta,r})^{-1}$, respectively.
\end{itemize}
%
%}}}
\end{appendix}

\section*{Acknowledgements}

We thank Harriet Bass and the referees for helpful comments.
The work of the first author was supported in part by an EPSRC
advance research fellowship and a Winton Capital research award.
The second author was supported by the Academy of Finland
Project 250575, by the Finnish Academy of Science and Letters, Vilho,
Yrj\"{o} and Kalle V\"{a}is\"{a}l\"{a} Foundation, by the Finnish
Centre of Excellence
in Analysis and Dynamics Research, and by the Finnish Doctoral
Programme in Stochastics and Statistics.

%suskaldyti doi

% imsref loaded by jurgita.kaciuliene, 2013-01-30 11:11:42
% imsref loaded by jurgita.kaciuliene, 2014-01-08 12:51:19

\printhistory

\end{document}